\documentclass[12pt, reqno]{amsart}
\usepackage{amsmath, amsfonts, amssymb, url}
\usepackage{eucal}
\usepackage{mathrsfs}
\usepackage{latexsym}
\usepackage{cite}
\usepackage{cases}
\allowdisplaybreaks[4]
\newtheorem{thm}{Theorem}[section]

\newtheorem{lem}[thm]{Lemma}

\theoremstyle{remark}

\theoremstyle{definition}

\newcommand{\nbd}{\nobreakdash}

\begin{document}
\title
[Mean width inequalities of sections and projections]{Mean width inequalities of sections and projections of convex bodies for isotropic measures}

\author[A.-J. Li]{Ai-Jun Li}
\address[A.-J. Li]{School of Science,
	Zhejiang University of Science and Technology, Hangzhou 310023, China}
\email{liaijun72@163.com}
\author[Q. Huang]{Qingzhong Huang}
\address[Q. Huang]{College of data science, Jiaxing University, Jiaxing 314001,	China} \email{hqz376560571@163.com}

\begin{abstract}
In this paper, we establish mean width inequalities of sections and
projections of convex bodies for isotropic measures with complete
equality conditions, which extends the recent work of
Alonso-Guti\'{e}rrez and Brazitikos. Different from their approach,
our proof is based on the approach developed by Lutwak, Yang and
Zhang, by using the Ball-Barthe inequality, the mass transportation,
and the isotropic embedding.
\end{abstract}



\thanks{The first author was supported by  Zhejiang Provincial Natural Science Foundation of China (LY22A010001). The second author was supported by the National Natural Science Foundation of China (No. 11701219). }

\date{}

\maketitle

\numberwithin{equation}{section}

\section{Introduction}

In  extremal problems related to the mean width of convex bodies, Euclidean  balls, cubes (or cross-polytopes), and simplices  are usually the extremizers.  An important example is the Urysohn inequality (e.g., \cite[p. 382]{Schneider}) which expresses the geometric fact that  Euclidean  balls minimize the mean width of convex bodies of given volume. Another example is the mean width inequality \cite{Barthe,Schechtman,Schmuckenschlager} that simplices have the extremal mean width of convex bodies in the John (L\"{o}wner)  position, whereas cubes and cross-polytopes  are the extremizers in the symmetric cases.  For more information about the mean width, see, e.g., \cite{Alonso-P,Milman,Alonso-P-1,LL,Boroczky-1,Fleury,Giannopoulos-2,Giannopoulos-3,Boroczky-0,Boroczky-2}.
In this paper, we will focus on the mean width of lower-dimensional sections and projections of convex bodies. The question of estimating sections and projections of convex bodies is always attractive, because it is not easy to give their sharp bounds (even for volumes) for  general convex bodies without any additional assumption. But for some special convex bodies (e.g., cubes, $l_p^n$-balls, and simlices) and  special positions (e.g.,  John   and isotropic positions), some remarkable results for the  lower dimensional sections and projections have been 
intensively investigated, for example,    \cite{Li-H-X,LiHX,Ivanov,Eskenazis,Barthe-2,Barthe-6,Fradelizi,CGL,Brzezinski,Ivanov-0,V.Milman,Hensley,Ball-1,Ball-02,Koldobsky,Meyer,Webb,Dirksen,Markessinis,Alonso-S}.

To introduce the mean width of a convex body, let $\langle \cdot,\cdot\rangle$ and $\|\cdot\|$ denote the scalar product and Euclidean norm in $n$-dimensional Euclidean space $\mathbb{R}^n$. Denote by $B_2^n$ and $S^{n-1}$ the Euclidean unit ball and its boundary in $\mathbb{R}^n$,  respectively. The volume of $B_2^n$ is 
solely written by $\omega_n$. We also denote by $B_1^n$ and $B_\infty^n$ the unit cross-polytope (the $l_1^n$-ball) and the unit cube (the $l_\infty^n$-ball) in $\mathbb{R}^n$. A convex body $K$ in $\mathbb{R}^n$ is a compact convex set with non-empty interior. Its support function $h_K(\cdot): \mathbb{R}^n\rightarrow
\mathbb{R}$ is defined for $x\in\mathbb{R}^n$ by $h_K(x)=\max
\{\langle x,y\rangle: y\in K\}.$ The mean width of $K$ is defined by
\begin{equation*}
	W(K)=\frac{1}{n\omega_n}\int_{S^{n-1}}(h_K(u)+h_K(-u))du
	=\frac{2}{n\omega_n}\int_{S^{n-1}}h_K(u)du,
\end{equation*}
where $du$ is the $(n-1)$-dimensional Hausdorff
measure (that coincides with the spherical Lebesgue measure in this case) with total mass $n\omega_n$.  If convex body $K$  contains the origin $o$ in
its interior, then its polar body $K^{\circ}$  is defined by
$K^{\circ}=\{x\in\mathbb{R}^n: \langle x,y\rangle\leq1 \ \textrm{for all} \  y\in K\}.$

To state the  mean width inequality,  we shall first introduce the concept of  John and 
L\"{o}wner positions. It is well-known that there is  a unique maximal volume ellipsoid (called the John ellipsoid) 
contained in a convex body $K$, or a unique minimal volume ellipsoid (called the L\"{o}wner ellipsoid) containing $K$.  We say that $K$ is in the John (L\"{o}wner) position if the John (L\"{o}wner) ellipsoid is the Euclidean unit ball $B_2^n$. 
In 1948, Fritz John \cite{John} showed that a convex body $K$  is in the John
(L\"{o}wner) position  if and
only if   for some $m\geq n$ there are unit
vectors $(u_i)_1^m$ on the boundary of $K$ and positive numbers
$(c_i)_1^m$ satisfying 
\begin{equation}\label{10-1}
\sum_{i=1}^mc_iu_i=0	
\end{equation}
and
\begin{equation}\label{10}
	\sum_{i=1}^mc_iu_i\otimes u_i=I_n,
\end{equation}
where $u_i\otimes u_i$ is the rank-one orthogonal projection onto
the space spanned by $u_i$ and $I_n$ is the identity map on
$\mathbb{R}^n$. Notice that  $(u_{i})_{1}^{m}$ are not concentrated on any hemisphere due to condition \eqref{10-1} and condition \eqref{10} guarantees that the $(u_{i})_{1}^{m}$ do not all lie close to a proper subspace of $\mathbb{R}^n$.
Let $\triangle_n$ denote the regular simplex in $\mathbb{R}^n$ inscribed into $B_2^n$. Its polar $\triangle_n^\circ$ is still a regular simplex with inradius $1$. Obviously,  the unit cross-polytope $B_1^n$ and  $\triangle_n$  are in the L\"{o}wner position. The unit cube $B_\infty^n$ and the polar $\triangle_n^\circ$ are in the John position.

An important application of  the John (L\"{o}wner) position  due to Ball \cite{Ball-03, Ball-02} is that condition \eqref{10}
can be perfectly combined with the Brascamp-Lieb inequality, which nowadays is
called as the geometric Brascamp-Lieb inequality.  Later, the reverse
Brascamp-Lieb inequality was proposed and established by Barthe
\cite{Barthe-4}. The geometric Brascamp-Lieb inequality and its dual
play an important role in establishing reverse (affine) isoperimetric
inequalities, that usually have simplices or,
in the symmetric case, cubes and their polars, as extremals. The  work of Ball and Barthe has motivated a series of new studies (see, e.g., \cite{Ball-04, Gordon, Gruber-2, 
	Lutwak-3,Alonso-S,Ball-02,Barthe,Barthe-5,Boroczky-0,Giannopoulos-1,LL,Li-H-4,Li-H-X,Li-X-H,Li3,Li5,LiHX,Lutwak-2,Lutwak-1,Lutwak-3,Lutwak-10,Schuster}). In particular, by the geometric
Brascamp-Lieb inequality and its dual, Barthe \cite{Barthe},
Schmuckenschl\"{a}ger \cite{Schmuckenschlager} established the
following mean width inequalities (see also Schechtman and Schmuckenschl\"{a}ger \cite{Schechtman}).
\begin{thm}	\label{t-1}
	Let $K$ be a convex body in $\mathbb{R}^n$. If $K$ is in L\"{o}wner position, then $W(K)\geq W(\triangle_n)$, with equality if and only if $K=\triangle_n$. If  $K$ is in John position, then $W(K)\leq W(\triangle_n^\circ)$, with equality  if and only if $K=\triangle_n^\circ$. 
	
	In symmetric cases,  if  $K$ is in L\"{o}wner position, then $W(K)\geq W(B_1^n)$, with equality if and only if $K=B_1^n$. If  $K$ is in John position, then $W(K)\leq W(B_\infty^n)$, with equality if and only if $K=B_\infty^n$.
	\end{thm}

 Let $G_{n,k}$ be the Grassmann manifold of $k$-dimensional linear subspaces in $\mathbb{R}^n$, $1\leq k\leq n$. Denote by $\mathrm{P}_{H}$  the orthogonal projection onto the subspace $H\in G_{n,k}$.  Denote by $\triangle_k$ the $k$-dimensional regular simplex inscribed into $B_2^k$ in $\mathbb{R}^k$ and by $\triangle_k^\circ$ the polar of $\triangle_k$  (where the polar operation is taken in $\mathbb{R}^k$). We  still denote the mean width of a convex body in $\mathbb{R}^k$  by $W(\cdot)$. Very recently, 
 Alonso-Guti\'{e}rrez and Brazitikos \cite{Alonso-S}  gave some estimates for the  mean width of sections and projections of convex bodies in the John (L\"{o}wner) position as follows. The case $k=n$ of   Theorem \ref{t-2} was proved in \cite{Barthe, Schmuckenschlager, Schechtman}.
 \begin{thm}	\label{t-2}
	Let $H\in G_{n,k}$ and $K$ be a convex body in $\mathbb{R}^n$. If $K$ is in L\"{o}wner position, then $W(\mathrm{P}_HK)\geq \sqrt{\frac{k}{n}}W(\triangle_k).$ If  $K$ is in John position, then
$W(K\cap H)\leq C\frac{n}{k}\sqrt{\frac{\log n}{\log k}}W(\triangle_k^\circ),$
where $C$ is an absolute constant. 	  
	
Furthermore,	let  $K$ be an origin-symmetric convex body in $\mathbb{R}^n$. If $K$ is in L\"{o}wner position, then $W(\mathrm{P}_HK)\geq \sqrt{\frac{k}{n}}W(B_1^k).$ If  $K$ is in John position, then $W(K\cap H)\leq \sqrt{\frac{n}{k}}W(B_\infty^k).$ 
\end{thm}

A finite Borel measure $\mu$ on  $S^{n-1}$
is said to be isotropic if
\begin{equation}\label{equ-1}
	\int_{S^{n-1}}u\otimes ud\mu(u)=I_n.
\end{equation}
Note that it is impossible for an isotropic measure to be
concentrated on a proper subspace of $\mathbb{R}^n$. Denote by $\mathrm{supp}\,\mu$ the support set of $\mu$. The  centroid of the measure $\mu$ is defined as
\begin{equation*}
	\frac{1}{\mu(S^{n-1})}\int_{S^{n-1}}ud\mu(u).
\end{equation*} The measure
$\mu$ is said to be even if it assumes the same value on antipodal
sets. Condition \eqref{equ-1}
reduces to \eqref{10} if the isotropic measure $\mu$ is of the
form $\sum_{i=1}^mc_i\delta_{u_i}$ or $(1/2)\sum_{i=1}^m(c_i\delta_{u_i}+c_i\delta_{-u_i})$ on
$S^{n-1}$ ($\delta_{x}$ stands for the Dirac mass at $x$).

 The aim of this paper is to establish  the following mean width inequalities of sections and projections of convex bodies for isotropic measures with complete equality conditions. 

\begin{thm}\label{thm-2}
	Let $\mu$ be an 
	isotropic measure on $S^{n-1}$ whose centroid is at
	the origin, and let $K$ be a convex body in $\mathbb{R}^n$ such that  $K$ contains the convex hull of $\mathrm{supp}\,\mu$; i.e., $K\supseteq \mathrm{conv}\{\mathrm{supp}\,\mu\}$. For any given $H\in
	G_{n,k}$,  we have
	$W(\mathrm{P}_HK)
		\geq\sqrt{\frac{k}{n}}W(\triangle_k)$
	with equality if and
	only if  $\mathrm{P}_HK=\sqrt{\frac{k}{n}}\triangle_k$.	
	
	Furthermore, if $\mu$ is even isotropic, then for any given $H\in
	G_{n,k}$, 	$W(\mathrm{P}_HK)\geq \sqrt{\frac kn}W(B_{1}^k)$
	with equality if and only if $\mathrm{P}_HK=\sqrt{\frac kn}B_{1}^k$; and 
	$W(K^{\circ}\cap H)\leq\sqrt{\frac nk}W(B_{\infty}^k)$
	with equality if and only if $K^\circ\cap H=\sqrt{\frac nk}B_\infty^k$.
\end{thm}

The case  $k=n$ of Theorem \ref{thm-2} was established in \cite{LL}. When $\mu$ is  discrete,  Theorem \ref{thm-2} is stated  in Theorem \ref{t-2} (without euqality conditions), and  reduces to Theorem \ref{t-1} for $k=n$. 

Note that the proof of Theorem \ref{thm-2} is based on the
approach developed by Lutwak, Yang, and Zhang
\cite{Lutwak-1,Lutwak-2}, which extends the work of Ball
\cite{Ball-03} and Barthe \cite{Barthe, Barthe-4} by using the
techniques of mass transportation, isotropic  embedding, and the Ball–Barthe inequality. For more applications about this
approach, see e.g.,
\cite{LiHX,Li-H-X,Schuster,Li3,Li-H-4,Li5,LL,Li-X-H}.


   We also note that the proof of the asymmetric case highly relies on  Lemma \ref{l-1}, which is  inspired from the idea of  Alonso-Guti\'{e}rrez and Brazitikos \cite{Alonso-S}. 

The rest of this paper is organized as follows: In Section 2 some of
the basic notations and preliminaries are listed. Section 3 provides some lemmas. The proof of Theorem \ref{thm-2} is presented in Sections 4 and 5. 

\section{Notations and preliminaries}
We list some basic facts about convex bodies. As general references
we recommend the books of Gardner \cite{Gardner} and Schneider
\cite{Schneider}.

We shall use $\|\cdot\|$ to denote the standard Euclidean norm on
$\mathbb{R}^n$. Write $\{e_1,\ldots, e_n\}$ for the standard
orthonormal basis of $\mathbb{R}^n$. A convex body $K$ in $\mathbb{R}^n$ is a compact convex set with non-empty
interior whose  support function, $h_K(\cdot): \mathbb{R}^n\rightarrow
\mathbb{R}$ is defined for $x\in\mathbb{R}^n$ by $$h_K(x)=\max
\{\langle x,y\rangle: y\in K\},$$ where $\langle x,y\rangle$ denotes the standard inner product of $x$ and $y$ in $\mathbb{R}^n$. 
If convex body $K$  contains the origin $o$ in
its interior, then its polar body $K^{\circ}$  is defined by
$$K^{\circ}=\{x\in\mathbb{R}^n: \langle x,y\rangle\leq1 \ \textrm{for all} \  y\in K\}.$$ The Minkowski functional
$\|\cdot\|_K$ of  $K$ is defined by
\begin{equation}\label{equ-2}
	\|x\|_K=\min\{t>0: x\in tK\}.
\end{equation}
In this case,
\begin{equation}\label{equ-3}
	\|x\|_K=h_{K^{\circ}}(x).
\end{equation}

Obviously, $h_K$ is 
homogeneous of degree $1$  if $K$ contains the origin in its interior. Integrating it with respect to the  Gaussian 
probability measure $\gamma_n$ with the
density $\frac{1}{(\sqrt{2\pi})^n}e^{-\|x\|^2/2}$, by polar coordinates, we have
\begin{align}\label{e-11}
	\int_{\mathbb{R}^n}h_K(x)d\gamma_n(x)&=\int_{0}^\infty 
	r^n\frac{e^{-\frac{r^2}{2}}}{(\sqrt{2\pi})^n}\int_{S^{n-1}}h_K(u)dudr\notag\\
	&=\frac{2c_n}{n\omega_n}\int_{S^{n-1}}h_K(u)du=c_nW(K),
\end{align}
where $c_n=\frac{\Gamma(\frac{n+1}{2})}{\sqrt{2}\pi^{\frac n2}}$.
We note that the mean width of  $K^\circ$  is equivalent to an important quantity in the local theory of Banach spaces---the
$\ell$-norm of $K$ (up to a factor); i.e., by \eqref{equ-3}, 
\begin{equation}\label{e-12}
	\ell(K):=\int_{\mathbb{R}^n}\|x\|_Kd\gamma_n(x)=\int_{\mathbb{R}^n}h_{K^\circ}(x)d\gamma_n(x)=c_nW(K^\circ).
\end{equation}
Moreover,
\begin{align}\label{equ-7}
	W(K)&=\frac{1}{c_n}\int_{\mathbb{R}^n}\|x\|_{K^\circ}d\gamma_n(x)
	=\frac{1}{c_n}\int_{\mathbb{R}^n}\Big(\int_0^{\|x\|_{K^\circ}}dt\Big)d\gamma_n(x)\notag\\
	&=\frac{1}{c_n}\int_0^{\infty}\int_{\mathbb{R}^n}\textbf{1}_{\{\|x\|_{K^\circ}>t\}}(x)d\gamma_n(x)dt
	=\frac{1}{c_n}\int_0^{\infty}[1-\gamma_n(tK^\circ)]dt,
\end{align}
where $\textbf{1}_A(\cdot)$ is the characteristic function of a set $A\subset\mathbb{R}^n$.

If $\mu$ is an isotropic measure on $S^{n-1}$ whose centroid is at
the origin, then taking the trace in \eqref{equ-1} yields
\begin{equation}\label{equ-5}
	\mu(S^{n-1})=n.
\end{equation}
Moreover, for any $H\in G_{n,k}$,
\begin{equation}\label{e-2}
	I_H=\int_{S^{n-1}\setminus H^{\perp}}\textrm{P}_Hu\otimes 
	\textrm{P}_Hud\mu(u)=\int_{S^{n-1}\setminus H^{\perp}}
	\frac{\textrm{P}_Hu}{\|\textrm{P}_Hu\|}\otimes 
	\frac{\textrm{P}_Hu}{\|\textrm{P}_Hu\|}\|\textrm{P}_Hu\|^2d\mu(u),
\end{equation}
where $I_H$ denotes the identity in  $H$, and
\begin{equation}\label{e-3}
	\int_{S^{n-1}\setminus H^\perp}\textrm{P}_Hud\mu(u)=\textrm{P}_H\Big(\int_{S^{n-1}}ud\mu(u)\Big)=o.
\end{equation}
Taking the trace in \eqref{e-2} yields
\begin{equation}\label{e-22}
	\int_{S^{n-1}\setminus H^{\perp}}
	\|\textrm{P}_Hu\|^2d\mu(u)=k.
\end{equation}

The following  Ball-Barthe inequality was
established by Lutwak, Yang, and Zhang \cite{Lutwak-1}, extending
the discrete case due to Ball and Barthe \cite{Barthe-4}.
\begin{lem}\label{l-6}
	Let $1\leq k\leq n$ and $H\in G_{n,k}$.	If $\nu$ is an isotropic measure on $S^{n-1}\cap H$ and
	$f$ is a positive continuous function on $\mathrm{supp}\,\nu$, then
	\begin{equation}\label{e-5}
		\det\int_{S^{n-1}\cap H}f(w)w\otimes wd\nu(w)\geq\exp\Big\{\int_{S^{n-1}\cap H}\log
		f(w)d\nu(w)\Big\},
	\end{equation}
	with equality if and only if $f(w_1)\cdots f(w_{k})$ is constant for
	linearly independent unit vectors $w_1,\ldots, w_{k} \in
	\mathrm{supp}\,\nu$.
\end{lem}

If $\nu$ is an isotropic measure on $S^{n-1}\cap H$, then let $L_p(\nu)$ denote
the space equipped with the standard $L_p$ norm of the function $f:
S^{n-1}\cap H\rightarrow\mathbb{R}$, with respect to $\nu$; i.e., for
$1\leq p<\infty$
$$\|f\|_{\nu,p}=\Big(\int_{S^{n-1}\cap H}|f(w)|^pd\nu(w)\Big)^{\frac1p}.$$
Obviously, for $f\in L_1(\nu)$,
$\int_{S^{n-1}\cap H}wf(w)d\nu(w)\in\mathbb{R}^k$. We shall need the
following lemma due to Lutwak, Yang, and Zhang \cite{Lutwak-2}.
\begin{lem}\label{lem-1}
	If $\nu$ is an isotropic measure on $S^{n-1}\cap H$
	and~$f\in L_2(\nu)$, then
	\begin{equation*}
		\Big\|\int_{S^{n-1}\cap H}w f(w)d\nu(w)\Big\|\leq   
		\Big(\int_{S^{n-1}\cap H}f(w)^2d\nu(w)
		\Big)^{1/2}.
	\end{equation*}
\end{lem}

If $\mu$ is an isotropic measure on $S^{n-1}$ whose centroid is at the origin, then we can define the
convex body $C$ in $\mathbb{R}^n$ as the convex hull of the support
of $\mu$; i.e., $$C=\textrm{conv}\{\textrm{supp}\,\mu\}.$$ Since
$\mathrm{supp}\,\mu$ is not contained in a closed hemisphere of $S^{n-1}$, the
convex body $C$ must contain the origin in its interior. Then the polar
body $C^{\circ}$of $C$ is given by
$$C^{\circ}=\{x\in\mathbb{R}^n: \langle x,u\rangle\leq1 \ \textrm{for all} \ u\in \textrm{supp}\,\mu\}.$$
Obviously, $C\subseteq B_2^n$ and  $B_2^n\subseteq C^{\circ}$. 
For $H\in G_{n,k}$, we  have
\begin{align}\label{equ-38}
	C^{\circ}\cap H&=\Big\{y\in H: \langle y, u\rangle\leq
	1 \ \ \textrm{for all}\ \ u\in
	\textrm{supp}\,\mu\Big\}\notag\\
	&=\Big\{y\in H: \langle y, \textrm{P}_Hu\rangle\leq
	1 \ \ \textrm{for all}\ \ u\in
	\textrm{supp}\,\mu\setminus 
	H^{\perp}\Big\},
\end{align}
and
\begin{align}\label{e-21}
	\mathrm{P}_HC
	&=\mathrm{P}_H(\mathrm{conv} 
	\{\mathrm{supp}\,\mu\})=\mathrm{conv}\{\mathrm{P}_Hu: u\in 
	\mathrm{supp}\,\mu\}\notag\\
	&=\mathrm{conv}\big\{\mathrm{P}_Hu: u\in 
	\mathrm{supp}\,\mu\setminus 
	H^{\perp}\big\},
\end{align}
where $H^{\perp}$ is the orthogonal complement of $H$. The following identity is easy to check (or see \cite[0.38]{Gardner}):
\begin{equation}\label{e-28}
	C^{\circ}\cap H=(\mathrm{P}_HC)^{\circ},
\end{equation}
where the polar operation on the right is taken in $H$.

If $K$ is a convex body in $\mathbb{R}^n$ such that $K\supseteq C$, then  we have
$\mathrm{P}_HK\supseteq\mathrm{P}_HC$ and $K^\circ\cap H\subseteq C^\circ \cap H$. Thus, to prove Theorem \ref{thm-2}, it suffices to prove the following theorem. 
\begin{thm}\label{thm-1}
	Let $H\in
	G_{n,k}$ and let $\mu$ be an 
	isotropic measure on $S^{n-1}$ whose centroid is at
	the origin. Then
	$
	W(\mathrm{P}_HC)
	\geq\sqrt{\frac{k}{n}}W(\triangle_k)
	$
	with equality if and
	only if $\mathrm{P}_HC=\sqrt{\frac{k}{n}}\triangle_k$.
	
	If $\mu$ is even isotropic, then
	$
	W(\mathrm{P}_HC)\geq \sqrt{\frac kn}W(B_{1}^k)
	$
	with equality if and only if $\mathrm{P}_HC=\sqrt{\frac kn}B_{1}^k$; and 
	$
	W(C^{\circ}\cap H)\leq\sqrt{\frac nk}W(B_{\infty}^k)
	$
	with equality if and only if $C^\circ\cap H=\sqrt{\frac nk}B_\infty^k$.
\end{thm}

\section{Some lemmas}
Define a function $v: S^{n-1}\rightarrow
S^{n}\subset\mathbb{R}^{n+1}=\mathbb{R}^n\times\mathbb{R}$ by
\begin{equation}\label{equ-8}
	v(u)=\Big(-\frac{\sqrt{n}}{\sqrt{n+1}}u,\frac{1}{\sqrt{n+1}}\Big), \ \ \ \ u\in S^{n-1}.
\end{equation}
For $H\in G_{n,k}$, define the
$(k+1)$-dimensional subspace $H'=H\times\mathbb{R}\subset\mathbb{R}^{n+1}$ by
\begin{equation}\label{e-0}
	H'=\textrm{span}\{H, e_{n+1}\}.
\end{equation}


If $\mu$ is an isotropic measure on $S^{n-1}$ whose centroid is at
the origin, then for each $u\in \textrm{supp}\,\mu$, define $w(u)\in S^n\cap H'$ by
\begin{equation}\label{e-4}
	w(u)=\frac{\mathrm{P}_{H'}v(u)}{\|\mathrm{P}_{H'}v(u)\|}
	=\frac{1}{\|\mathrm{P}_{H'}v(u)\|}
	\Big(-\frac{\sqrt{n}}{\sqrt{n+1}}\mathrm{P}_{H}u,
	\frac{1}{\sqrt{n+1}}\Big),
\end{equation}
where $v(u)$ is defined in \eqref{equ-8}. Let $\nu$ be the Borel measure on $S^{n}\cap H'$ defined by
\begin{equation}\label{e-1}
	\int_{S^{n}\cap
		H'}g(w)d\nu(w)=\frac{n+1}{n}\int_{S^{n-1}}g
	\Big(\frac{\mathrm{P}_{H'}v(u)}{\|\mathrm{P}_{H'}v(u)\|}\Big)
	\|\mathrm{P}_{H'}v(u)\|^2d\mu(u),
\end{equation}
for any continuous
function $g: S^{n}\cap H'\rightarrow\mathbb{R}$.
The following lemma shows that the measure $\nu$ induced by
\eqref{e-1} is also isotopic on $S^n\cap H'$. In other words, the
continuous map $\mathrm{P}_{H'}v(\cdot): S^{n-1}\rightarrow H'$ is
an isotropic embedding of~$\mu$. The concept of isotropic embedding
was originated from Ball \cite{Ball-03}  and  introduced by Lutwak,
Yang, and Zhang \cite{Lutwak-10} (see also \cite{Schuster}).

\begin{lem}\label{l-3}
	If $\mu$ is an isotropic measure on $S^{n-1}$ whose centroid is at
	the origin, then the measure $\nu$ induced by \eqref{e-1} is
	isotropic on $S^n\cap H'$. Moreover,
	\begin{equation}\label{equ-10}
		\int_{S^n\cap H'}w \langle w,e_{n+1}\rangle d\nu(w)=e_{n+1}.
	\end{equation}
\end{lem}
\begin{proof}
	Let $z=(x, r)\in H'=H\times\mathbb{R}\subset\mathbb{R}^{n+1}$. From \eqref{e-1}, \eqref{equ-8},
	\eqref{e-3}, and \eqref{equ-5}, we get
	\begin{align*}
		\int_{S^n\cap H'}\langle
		z,w\rangle^2d\nu(w)
		&=\frac{n+1}{n}\int_{S^{n-1}}\bigg\langle(x,r),
		\frac{\mathrm{P}_{H'}v(u)}{\|\mathrm{P}_{H'}v(u)\|}\bigg\rangle^2\|\mathrm{P}_{H'}v(u)\|^2d\mu(u)\\
		&=\frac{n+1}{n}\int_{S^{n-1}}\big\langle(x,r),
		\mathrm{P}_{H'}v(u)\big\rangle^2d\mu(u)\\
		&=\frac{n+1}{n}\int_{S^{n-1}}\bigg\langle(x,r),
		\Big(-\frac{\sqrt{n}}{\sqrt{n+1}}\textrm{P}_{H}u,\frac{1}{\sqrt{n+1}}\Big)\bigg\rangle^2d\mu(u)\\
		&=\int_{S^{n-1}}\langle
		x,\textrm{P}_{H}u\rangle^2d\mu(u)-\frac{2r}{\sqrt{n}}\Big\langle
		x,\int_{S^{n-1}}\textrm{P}_{H}ud\mu(u)\Big\rangle+\frac{r^2}{n}\int_{S^{n-1}}d\mu(u)\\
		&=|x|^2+r^2\\
		&=|z|^2,
	\end{align*}
	which implies that the measure $\nu$ on $S^n\cap H'$ is isotropic.
	
	Moreover,  by \eqref{e-4}, \eqref{equ-8}, and \eqref{e-3},
	\begin{align*}
		&\Big\langle
		z,\int_{S^n\cap H'}w\langle w,e_{n+1}\rangle d\nu(w)\Big\rangle\\&=\frac{\sqrt{n+1}}{n}\int_{S^{n-1}}\bigg\langle(x,r),
		\frac{\mathrm{P}_{H'}v(u)}{\|\mathrm{P}_{H'}v(u)\|}\bigg\rangle\|\mathrm{P}_{H'}v(u)\| d\mu(u)\\&=\frac{\sqrt{n+1}}{n}\int_{S^{n-1}}\bigg\langle(x,r),
		\Big(-\frac{\sqrt{n}}{\sqrt{n+1}}\mathrm{P}_{H}u,\frac{1}{\sqrt{n+1}}\Big)\bigg\rangle d\mu(u)\\
		&=-\frac{1}{\sqrt{n}}\Big\langle
		x,\int_{S^{n-1}}\textrm{P}_Hud\mu(u)\Big\rangle+\frac{r}{n}\int_{S^{n-1}}d\mu(u)\\
		&=r\\
		&=\langle z, e_{n+1}\rangle,
	\end{align*}
	as desired.
\end{proof}
\par
Since $\nu$ is an isotropic measure on $S^n\cap H'$, it follows that
\begin{equation*}
	\int_{S^n\cap H'}w\otimes wd\nu(w)=I_{H'}.	
\end{equation*}
Taking the trace in the above equation yields 
\begin{equation}\label{equ-11}
	\nu(S^n\cap H')=k+1,
\end{equation}
or, by \eqref{e-1},
\begin{equation}\label{e-6}
	\int_{S^{n-1}}
	\|\mathrm{P}_{H'}v(u)\|^2d\mu(u)=\frac{n}{n+1}\int_{S^{n}\cap
		H'}d\nu(w)=\frac{n}{n+1}\nu(S^n\cap H')
	=\frac{n(k+1)}{n+1}.
\end{equation}

To prove Lemma \ref{l-1}, the following  Pr\'{e}kopa-Leindler 
inequality will be needed. We offer a simple proof based
on Gardner \cite[Theorem 7.1]{Gardner}.
\begin{lem}\label{l-2}
Let $\sigma$ be a Borel probability measure on $S^{n-1}$, and 
 let $(f_u(\cdot))_{u\in \mathrm{supp}\,\sigma}$ be 
a family of nonnegative integrable functions on $\mathbb{R}$ such that, 
for any functions $(\tau_u(\cdot))_{u\in \mathrm{supp}\,\sigma}:\mathbb{R}\rightarrow\mathbb{R}$ 
being integrable with respect to $\sigma$ and 
for a nonnegative integrable  function $h$ on $\mathbb{R}$,
\begin{equation}\label{e-30}
	h\Big(\int_{S^{n-1}}\tau_u(r)d\sigma(u)\Big)\geq
	\exp\Big(\int_{S^{n-1}}\log f_u(\tau_u(r))d\sigma(u)\Big).
\end{equation}
Then, 
\begin{equation*}
	\int_{\mathbb{R}}h(r)dr\geq\exp\Big(\int_{S^{n-1}}\log
	\Big(\int_{\mathbb{R}}f_u(r)dr\Big)d\sigma(u)\Big).
\end{equation*}
\end{lem}
\begin{proof}
It is sufficient to prove the assertion for continuous, positive,
integrable functions; the general case then follows by standard 
measure-theoretic arguments. For any fixed $u\in \mathrm{supp}\,\sigma$, let
	\begin{equation*}
		\int_{\mathbb{R}}f_{u}(r)dr=F_u>0.
	\end{equation*}
For $t\in(0, 1) $, define $\phi_u:(0,1)\rightarrow\mathbb{R}$  by
\begin{equation*}
	\frac{1}{F_u}\int_{-\infty}^{\phi_u(t)}f_u(r)dr=t.
\end{equation*}
The functions $\phi_u$ are increasing and differentiable, so the differential of $\phi_u$ gives
\begin{equation*}
	\frac{f_u(\phi_u(t))\phi_u'(t)}{F_u}=1.
\end{equation*}
The function $\phi_u'$ are continuous and therefore bounded on the sphere.
Let
\begin{equation*}
	T(t)=\int_{S^{n-1}}\phi_u(t)d\sigma(u).
\end{equation*}
By dominated convergence and the arithmetic-geometric
means inequality, we have
\begin{align*}
	T'(t)&=\int_{S^{n-1}}\phi_u'(t)d\sigma(u)\geq\exp\Big(\int_{S^{n-1}}
	\log \phi_u'(t)d\sigma(u)\Big)\\
	&=\exp\Big(\int_{S^{n-1}}
	\log \frac{F_u}{f_u(\phi_u(t))}d\sigma(u)\Big).
\end{align*}
Therefore, applying \eqref{e-30} with $\tau_u=\phi_u$ yields
\begin{align*}
	\int_{\mathbb{R}}h(r)dr&\geq\int_0^1h(T(t))T'(t)dt\\
	&\geq\int_0^1h\Big(\int_{S^{n-1}}\phi_u(t)d\sigma(u)\Big)\exp\Big(\int_{S^{n-1}}
	\log \frac{F_u}{f_u(\phi_u(t))}d\sigma(u)\Big)dt\\
	&\geq\int_0^1\exp\Big(\int_{S^{n-1}}\log 
	f_u(\phi_u(t))d\sigma(u)\Big)\exp\Big(\int_{S^{n-1}}
	\log \frac{F_u}{f_u(\phi_u(t))}d\sigma(u)\Big)dt\\
	&=\int_0^1\exp\Big(\int_{S^{n-1}}\log F_ud\sigma(u)\Big)dt\\
	&=\exp\Big(\int_{S^{n-1}}\log \Big(\int_{\mathbb{R}}f_u\Big)d\sigma(u)\Big),
\end{align*}
as desired.
\end{proof}

To prove the asymmetric case, the following crucial lemma,  inspired by the idea of  Alonso-Guti\'{e}rrez and Brazitikos \cite{Alonso-S}, will be needed.

\begin{lem}\label{l-1}
	Let $\lambda$ be an arbitrary real number, $H\in G_{n,k}$, and let $H'$
	be defined in \eqref{e-0}. 	If 
	\begin{align}\label{e-8}
		&\int_0^{\infty}e^{-\frac{r^2}{2}+(\lambda-\sqrt{n+1})r}
		\gamma_k\Big(\frac{r}{\sqrt{n}}(C^{\circ}\cap 
		H)\Big)dr\notag\\
		&\leq(2\pi)^{-\frac{k}{2}}\exp\bigg(\int_{S^n\cap H'}\log 
		\Big(\int_0^{\infty}e^{-\frac{r^2}{2}+\langle 
			w,(\lambda-\sqrt{n+1})e_{n+1}\rangle r}dr\Big)d\nu(w)\bigg),
	\end{align}
then we have,
	\begin{equation*}
			W(\mathrm{P}_HC)
		\geq\sqrt{\frac{k}{n}}W(\triangle_k),
	\end{equation*}	
where $\triangle_k$ is the $k$-dimensional regular simplex inscribed into $B_2^k$ in $H$.
\end{lem}
\begin{proof}
		Observe that inequality \eqref{e-8} is equivalent to 
	\begin{align}\label{e-9}
	&\int_0^{\infty}e^{-\frac{(r-(\lambda-\sqrt{n+1}))^2}{2}}
	\gamma_k\Big(\frac{r}{\sqrt{n}}(C^{\circ}\cap 
	H)\Big)dr\notag\\
	&\leq(2\pi)^{-\frac{k}{2}}\exp\bigg(\int_{S^n\cap H'}\log 
	\Big(\int_0^{\infty}e^{-\frac{(r-\langle 
			w,(\lambda-\sqrt{n+1})e_{n+1}\rangle)^2}{2}}dr\Big)
	d\nu(w)\bigg).		
	\end{align}
In fact, it is easy to see that 
\begin{equation*}
\int_0^{\infty}e^{-\frac{r^2}{2}+(\lambda-\sqrt{n+1})r}
\gamma_k\Big(\frac{r}{\sqrt{n}}(C^{\circ}\cap 
H)\Big)dr=
e^{\frac{(\lambda-\sqrt{n+1})^2}{2}}\int_0^{\infty}
e^{-\frac{(r-(\lambda-\sqrt{n+1}))^2}{2}}
\gamma_k\Big(\frac{r}{\sqrt{n}}(C^{\circ}\cap 
H)\Big)dr.	
\end{equation*}	
On the other hand, by  \eqref{e-1}, \eqref{e-4}, and \eqref{equ-5},
\begin{align*}
&\exp\bigg(\int_{S^n\cap H'}\log 
\Big(\int_0^{\infty}e^{-\frac{r^2}{2}+\langle 
	w,(\lambda-\sqrt{n+1})e_{n+1}\rangle r}dr\Big)d\nu(w)\bigg)\\
&=\exp\bigg(\int_{S^n\cap H'}\log 
\Big(e^{\frac{\langle 
		w,(\lambda-\sqrt{n+1})e_{n+1}\rangle^2}{2}}
	\int_0^{\infty}e^{-\frac{(r-\langle 
		w,(\lambda-\sqrt{n+1})e_{n+1}\rangle)^2}{2}}dr\Big)
	d\nu(w)\bigg)\\
	&=\exp\Big(\int_{S^n\cap H'}\frac{\langle 
			w,(\lambda-\sqrt{n+1})e_{n+1}\rangle^2}{2}
	d\nu(w)\Big)\\
	&\qquad\times\exp\bigg(\int_{S^n\cap H'}\log 
	\Big(\int_0^{\infty}e^{-\frac{(r-\langle 
			w,(\lambda-\sqrt{n+1})e_{n+1}\rangle)^2}{2}}dr\Big)
	d\nu(w)\bigg)\\
	&=\exp\Big(\frac{n+1}{n}\int_{S^{n-1}}\frac{(\lambda-\sqrt{n+1})^2}
	{2(n+1)\|\mathrm{P}_{H'}v(u)\|^2}\|\mathrm{P}_{H'}v(u)\|^2
	d\mu(u)\Big)\\
	&\qquad\times\exp\bigg(\int_{S^n\cap H'}\log 
	\Big(\int_0^{\infty}e^{-\frac{(r-\langle 
			w,(\lambda-\sqrt{n+1})e_{n+1}\rangle)^2}{2}}dr\Big)
	d\nu(w)\bigg)\\	
	&=e^{\frac{(\lambda-\sqrt{n+1})^2}{2}}\exp\bigg(\int_{S^n\cap H'}\log 
	\Big(\int_0^{\infty}e^{-\frac{(r-\langle 
			w,(\lambda-\sqrt{n+1})e_{n+1}\rangle)^2}{2}}dr\Big)
	d\nu(w)\bigg).	
\end{align*}
Canceling the factor  $e^{\frac{(\lambda-\sqrt{n+1})^2}{2}}$ yields  inequality \eqref{e-9}. 

Note that by \eqref{e-6} 
\begin{equation*}
	\frac{1}{k+1}\int_{S^{n}\cap H'}
	d\nu(w)=1.
\end{equation*}
By Jensen's inequality,  for any integrable functions $\tau_w$ with respect to $\nu$, we obviously have,
\begin{align*}
	&\frac{1}{k+1}\int_{S^{n}\cap H'}
	\frac{(\tau_w(r)-\langle 
		w,(\lambda-\sqrt{n+1})e_{n+1}\rangle)^2}{2}	
	d\nu(w)\\&\geq\frac{\Big(\frac{1}{k+1}\int_{S^{n}\cap H'}(\tau_w(r)-\langle 
		w,(\lambda-\sqrt{n+1})e_{n+1}\rangle)	
		d\nu(w)\Big)^2}{2},
\end{align*}
and thus, 
\begin{align}\label{e-16}
	&\exp\Bigg(-\frac{1}{k+1}\int_{S^{n}\cap H'}
	\frac{(\tau_w(r)-\langle 
		w,(\lambda-\sqrt{n+1})e_{n+1}\rangle)^2}{2}	
	d\nu(w)\Bigg)\notag\\
	&\leq\exp 	
	\bigg(-\frac{\Big(\frac{1}{k+1}\int_{S^{n}\cap H'}(\tau_w(r)-\langle 
		w,(\lambda-\sqrt{n+1})e_{n+1}\rangle)	
		d\nu(w)\Big)^2}{2}\bigg).
\end{align}
For $r\in\mathbb{R}^+$, let 
\begin{align*}
	h(r)=\bigg(-\frac{\Big(\frac{1}{k+1}\int_{S^{n}\cap H'}r-\langle 
		w,(\lambda-\sqrt{n+1})e_{n+1}\rangle	
		d\nu(w)\Big)^2}{2}\bigg),
\end{align*}
and 
\begin{equation*}
	f_w(r)=\exp \bigg(-\frac{(r-\langle 
		w,(\lambda-\sqrt{n+1})e_{n+1}\rangle)^2}{2}\bigg).
\end{equation*}
Then by \eqref{e-16} and 
 Lemma \ref{l-2}, we obtain 
\begin{align}\label{e-10}
	&\exp\Bigg(\frac{1}{k+1}\int_{S^{n}\cap H'}\log 
	\bigg(\int_0^{\infty}\exp \bigg(-\frac{(r-\langle 
		w,(\lambda-\sqrt{n+1})e_{n+1}\rangle)^2}{2}\bigg)dr\bigg)	
		d\nu(w)\Bigg)\notag\\
	&\leq	\int_0^{\infty}\exp 
	\Bigg(-\frac{\Big(\frac{1}{k+1}\int_{S^{n}\cap 
	H'}r-\langle 
w,(\lambda-\sqrt{n+1})e_{n+1}\rangle
		d\nu(w)\Big)^2}{2}\Bigg)dr.		
\end{align}

Let
\begin{equation*}
	\alpha_\lambda=\Big(\frac{\lambda}{\sqrt{n+1}}-1\Big), \ \ \beta=\frac{1}
	{\sqrt{k+1}}  	\int_{S^{n}\cap H'}\langle w,\sqrt{n+1}e_{n+1}\rangle
	d\nu(w).
\end{equation*}
Then by \eqref{e-1}, the Cauchy-Schwartz inequality, \eqref{equ-5}, and 
\eqref{e-6},
 we get  
\begin{align}\label{e-7}
\beta&=\frac{n+1}
{n\sqrt{k+1}}  \int_{S^{n-1}}\|\mathrm{P}_{H'}v(u)\|
d\mu(u)\notag\\&\leq \frac{n+1}{n\sqrt{k+1}}\Big( \int_{S^{n-1}}d\mu(u)\Big)^{\frac12}\Big(  \int_{S^{n-1}}\|\mathrm{P}_{H'}v(u)\|^2d\mu(u)\Big)^{\frac12}\notag\\
&=\sqrt{n+1}.	
\end{align}

Therefore, by \eqref{e-10}, we 
arrive 
at
\begin{align*}
&\exp\bigg(\int_{S^n\cap H'}\log 
\Big(\int_0^{\infty}e^{-\frac{(r-\langle 
		w,(\lambda-\sqrt{n+1})e_{n+1}\rangle)^2}{2}}dr\Big)	
d\nu(w)\bigg)\\
      &\leq	\left(\int_0^{\infty}\exp 
      \Bigg(-\frac{\Big(r-\frac{1}{k+1}\int_{S^{n}\cap 
      		H'}\langle 
      	w,(\lambda-\sqrt{n+1})e_{n+1}\rangle
      	d\nu(w)\Big)^2}{2}\Bigg)dr	\right)^{k+1}\\
            &=\Bigg(\int_0^{\infty}\exp 
      \bigg(-\frac{\big(r-\frac{\alpha_\lambda\beta}
      	{\sqrt{k+1}}  \big)^2}{2}\bigg)dr	\Bigg)^{k+1}\\
      &=\int_{[0,\infty)^{k+1}}\prod_{i=1}^{k+1}\exp 
      \bigg(-\frac{\big(x_i-\frac{\alpha_\lambda\beta}
      	{\sqrt{k+1}}  \big)^2}{2}\bigg)dx_1\cdots dx_{k+1}\\
      &=\int_{\big[-\frac{\alpha_\lambda\beta}
      	{\sqrt{k+1}} ,\infty\big)^{k+1}}e^{-\frac{|x|^2}{2}} 
      dx,
\end{align*}
where $x=(x_1,\dots,x_{k+1})\in\mathbb{R}^{k+1}$. Notice that, for the direction $v_0=\big(\frac{1}{\sqrt{k+1}},\dots,\frac{1}{\sqrt{k+1}}\big)$,
\begin{align*}
	&\Big\{x\in\Big[-\frac{\alpha_\lambda\beta}
{\sqrt{k+1}},\infty\Big)^{k+1}: \langle x, v_0\rangle=r\Big\}\\
&=\Big\{x\in[-\alpha_\lambda\beta,\infty)^{k+1}: \Big\langle \frac{x}{\sqrt{k+1}}, v_0\Big\rangle=r\Big\}\\
&=\Big\{x\in[0,\infty)^{k+1}: \Big\langle \frac{x}{\sqrt{k+1}}, v_0\Big\rangle=r+\alpha_\lambda\beta\Big\}-\alpha_\lambda\beta(1,\dots,1)\\
&=\Big\{x\in[0,\infty)^{k+1}: \langle x, v_0\rangle=r+\alpha_\lambda\beta\Big\}-\alpha_\lambda\beta(1,\dots,1)\\
&=(r+\alpha_\lambda\beta)\sqrt{k+1}\mathrm{conv}\{e_1,\dots,e_{k+1}\}-\alpha_\lambda\beta(1,\dots,1).
\end{align*}
By the rotation invariance of Gaussian 
measures, we may integrate in the direction $v_0$,  make the change of variables $\langle x,v_0\rangle=r$, and integrate in the hyperplane $\{x\in\mathbb{R}^{k+1}: \langle x, v_0\rangle=r\}$ as follows. 
\begin{align*}
	&\int_{\big[-\frac{\alpha_\lambda\beta}
		{\sqrt{k+1}} ,\infty\big)^{k+1}}e^{-\frac{|x|^2}{2}}
	dx\\&=\int_{\big[-\frac{\alpha_\lambda\beta}
		{\sqrt{k+1}} ,\infty\big)^{k+1}}e^{-\frac{-|\mathrm{P}_{v_0^{\perp}}x|^2-|\mathrm{P}_{v_0}x|^2}{2}}
	dx\\
	&=(2\pi)^{\frac{k}{2}}\int_{-\alpha_\lambda\beta}^{\infty}
	\gamma_k\Big(\mathrm{P}_{v_0^{\perp}}\Big\{x\in\Big[-\frac{\alpha_\lambda\beta}
	{\sqrt{k+1}},\infty\Big)^{k+1}: \langle x, v_0\rangle=r\Big\}\Big)e^{-\frac{r^2}{2}}dr\\
	&=(2\pi)^{\frac{k}{2}}\int_{-\alpha_\lambda\beta}^\infty
	e^{-\frac{r^2}{2}}
	\gamma_k\big((r+\alpha_\lambda\beta)\sqrt{k+1}
	\mathrm{conv}\{e_1,\dots,e_{k+1}\}\big)dr\\
	&=(2\pi)^{\frac{k}{2}}\int_{-\alpha_\lambda\beta}^\infty
	e^{-\frac{r^2}{2}}
	\gamma_k\Big((r+\alpha_\lambda\beta)
	\frac{1}{\sqrt{k}}\triangle_k^\circ\Big)dr\\
	&=(2\pi)^{\frac{k}{2}}\int_{0}^\infty
	e^{-\frac{(r-\alpha_\lambda\beta)^2}{2}}
	\gamma_k\Big(\frac{r}{\sqrt{k}}\triangle_k^\circ\Big)dr,
\end{align*}
where we use the fact that  
$\gamma_k(\sqrt{k(k+1)}\mathrm{conv}\{e_1,\dots,e_{k+1}\})=\gamma_k(\triangle_k^\circ)$.

Thus, we have proved that
\begin{equation*}
\int_0^{\infty}e^{-\frac{(r-\alpha_\lambda\sqrt{n+1})^2}{2}}
\gamma_k\Big(\frac{r}{\sqrt{n}}(C^{\circ}\cap 
H)\Big)dr\leq	\int_{0}^\infty 
e^{-\frac{(r-\alpha_\lambda\beta)^2}{2}} 
\gamma_k\Big(\frac{r}{\sqrt{k}}\triangle_k^\circ\Big)dr.
\end{equation*}
Applying the above inequality to $-\alpha_\lambda$ also gives
\begin{equation*}
	\int_0^{\infty}e^{-\frac{(r+\alpha_\lambda\sqrt{n+1})^2}{2}}
	\gamma_k\Big(\frac{r}{\sqrt{n}}(C^{\circ}\cap 
	H)\Big)dr\leq	\int_{0}^\infty 
	e^{-\frac{(r+\alpha_\lambda\beta)^2}{2}} 
	\gamma_k\Big(\frac{r}{\sqrt{k}}\triangle_k^\circ\Big)dr,
\end{equation*}
or, equivalently, 
\begin{equation*}
	\int_{-\infty}^0e^{-\frac{(r-\alpha_\lambda\sqrt{n+1})^2}{2}}
	\gamma_k\Big(\frac{|r|}{\sqrt{n}}(C^{\circ}\cap 
	H)\Big)dr\leq	\int_{-\infty}^0 
	e^{-\frac{(r-\alpha_\lambda\beta)^2}{2}} 
	\gamma_k\Big(\frac{|r|}{\sqrt{k}}\triangle_k^\circ\Big)dr.
\end{equation*}
Therefore, for any $\alpha_\lambda\in\mathbb{R}$,
\begin{equation*}
	\int_{-\infty}^{\infty}e^{-\frac{(r-\alpha_\lambda\sqrt{n+1})^2}{2}}
	\gamma_k\Big(\frac{|r|}{\sqrt{n}}(C^{\circ}\cap 
	H)\Big)dr\leq	\int_{-\infty}^{\infty} 
	e^{-\frac{(r-\alpha_\lambda\beta)^2}{2}} 
	\gamma_k\Big(\frac{|r|}{\sqrt{k}}\triangle_k^\circ\Big)dr,
\end{equation*}
which is equivalent to for all $\alpha_\lambda\in\mathbb{R}$,
\begin{equation*}
\int_{-\infty}^{\infty}e^{-\frac{(r-\alpha_\lambda\sqrt{n+1})^2}{2}}
\bigg(1-\gamma_k\Big(\frac{|r|}{\sqrt{n}}(C^{\circ}\cap 
H)\Big)\bigg)dr\geq	\int_{-\infty}^\infty 
e^{-\frac{(r-\alpha_\lambda\beta)^2}{2}} 
\bigg(1-\gamma_k\Big(\frac{|r|}{\sqrt{k}}\triangle_k^\circ\Big)\bigg)dr.	
\end{equation*}
Integrating with respect to $\alpha_\lambda$, we get
\begin{equation*}
\frac{1}{\sqrt{n+1}}\int_{-\infty}^{\infty}
\bigg(1-\gamma_k\Big(\frac{|r|}{\sqrt{n}}(C^{\circ}\cap 
H)\Big)\bigg)dr\geq	\frac{1}{\beta}\int_{-\infty}^\infty 
\bigg(1-\gamma_k\Big(\frac{|r|}{\sqrt{k}}\triangle_k^\circ\Big)\bigg)dr.
\end{equation*}
Equivalently,
\begin{equation*}
	\frac{1}{\sqrt{n+1}}\int_{0}^{\infty}
	\bigg(1-\gamma_k\Big(\frac{r}{\sqrt{n}}(C^{\circ}\cap 
	H)\Big)\bigg)dr\geq	\frac{1}{\beta}\int_{0}^\infty 
	\bigg(1-\gamma_k\Big(\frac{r}{\sqrt{k}}\triangle_k^\circ\Big)\bigg)dr,	
\end{equation*}
or
\begin{equation*}
	\sqrt{\frac{n}{n+1}}\int_{0}^{\infty}
	\bigg(1-\gamma_k\Big(r(C^{\circ}\cap 
	H)\Big)\bigg)dr\geq	\frac{\sqrt{k}}{\beta}\int_{0}^\infty 
	\bigg(1-\gamma_k\big(r\triangle_k^\circ\big)\bigg)dr,	
\end{equation*}
Thus, by \eqref{e-28}, \eqref{equ-7}, and \eqref{e-7},  we obtain
\begin{equation*}
	W(\mathrm{P}_HC)=W((C^\circ\cap 
	H)^\circ)\geq\frac{1}{\beta}\sqrt{\frac{(n+1)k}{n}}W(\triangle_k)
	\geq\sqrt{\frac{k}{n}}W(\triangle_k).
\end{equation*}
\end{proof}

\section{The asymmetric case}

\begin{thm}\label{thm-4-1}
		Let $H\in G_{n,k}$ and let
	$\triangle_k$ be a regular $k$-simplex inscribed in $B_2^k$. If $\mu$ is an 
	isotropic measure on $S^{n-1}$ whose centroid is at
	the origin, then
\begin{equation*}
	W(\mathrm{P}_HC)
	\geq\sqrt{\frac{k}{n}}W(\triangle_k),
\end{equation*}	
where $C=\mathrm{conv}\{\mathrm{supp}\,\mu\}.$ There is equality if and
	only if $\mathrm{P}_HC=\sqrt{\frac{k}{n}}\triangle_k$.
\end{thm}
\begin{proof}
Since $\mu$ is an isotropic measure on $S^{n-1}$ whose centroid is at
the origin, it follows from Lemma \ref{l-3} that  the measure $\nu$ induced by 
\eqref{e-1} is isotropic on $S^n\cap H'$, where 
$H'=\textrm{span}\{H,e_{n+1}\}.$

Let $\lambda$ be an arbitrary real number. For $w\in \mathrm{supp}\,\nu$,	let
	\begin{equation}\label{e-20}
		G_{\lambda,w}
		=\int_0^{\infty}e^{-\frac{s^2}{2}+\langle 
		w,(\lambda-\sqrt{n+1})e_{n+1}\rangle s}ds.
	\end{equation}
	Define the smooth strictly increasing function $\phi_{\lambda,w}:
	(0,\infty)\rightarrow\mathbb{R}$ by
	\begin{equation*}
		\frac{1}{G_{\lambda,w}}\int_0^{t}e^{-\frac{s^2}{2}+\langle 
			w,(\lambda-\sqrt{n+1})e_{n+1}\rangle s}ds
		=\frac{1}{\sqrt{\pi}}\int_{-\infty}^{\phi_{\lambda,w}(t)}e^{-s^2}ds.
	\end{equation*}
Then  $\phi_{\lambda,w}'>0$ and differentiating both sides of the above identity yields, for all $t>0$, 
	\begin{equation}\label{equ-12}
		-\frac{t^2}{2}+\big\langle w,(\lambda-\sqrt{n+1})e_{n+1}\big\rangle 
		t=\log
		G_{\lambda,w}-\log\sqrt{\pi}-\phi_{\lambda,w}(t)^2+
		\log\phi_{\lambda,w}'(t).
	\end{equation}
	\par
	Define an open cone $\mathcal
	{D}\subset H'=H\times\mathbb{R}$ by
	\begin{equation*}
		\mathcal
		{D}=\bigcup_{r>0}\Big(\frac{r}{\sqrt{n}}\textrm{relint} (C^{\circ}\cap H)\Big)\times\{r\}.
	\end{equation*}
	It is easy to see that $z=(x, r)\in\mathcal {D}$ if and only if
	$\frac{\sqrt{n}}{r}x\in \textrm{relint} (C^{\circ}\cap H)$. By the
	definition \eqref{equ-38} of $C^{\circ}\cap H$, it is equivalent to
	$$\langle x, \mathrm{P}_Hu\rangle<\frac{r}{\sqrt{n}}
	 \ \ \textrm{for all} \ \ u\in \textrm{supp}\,\mu.$$
	From \eqref{e-4} we have
	\begin{align*}
		\langle z,w(u)\rangle&=\bigg\langle (x,r),
		\frac{1}{\|\mathrm{P}_{H'}v(u)\|}
		\Big(-\frac{\sqrt{n}}{\sqrt{n+1}}\mathrm{P}_{H}u,\frac{1}
		{\sqrt{n+1}}\Big)\bigg\rangle\\
		&=\frac{1}{\|\mathrm{P}_{H'}v(u)\|}\Big(-\frac{\sqrt{n}}{\sqrt{n+1}}
		\langle	x,\mathrm{P}_{H}u\rangle+\frac{r}{\sqrt{n+1}}\Big)\\
		&>\frac{1}{\|\mathrm{P}_{H'}v(u)\|}\Big(-\frac{r}{\sqrt{n+1}}
		+\frac{r}{\sqrt{n+1}}\Big)\\
		&=0,
	\end{align*}
for  $u\in \textrm{supp}\,\mu.$
	Then we obtain that
	\begin{equation}\label{equ-13}
		z\in\mathcal {D} \ \ \Longleftrightarrow \ \  \langle z, w(u)\rangle>0 \
		\ \textrm{for all} \ \ u\in \textrm{supp}\,\mu.
	\end{equation}
	\par
For the isotropic measure $\nu$ on $S^n\cap H'$ induced by \eqref{e-1},
	define a transformation $T_1:\mathcal {D}\rightarrow
	H'$ by
	\begin{equation}\label{equ-14}
		T_1z=\int_{S^n\cap H'}w\phi_{\lambda,w}(\langle
		z,w\rangle)d\nu(w),
	\end{equation}
	for each $z\in\mathcal {D}$. 
	The differential of $T_1$ is given by
	\begin{equation}\label{equ-16}
		dT_1(z)=\int_{S^n\cap H'}\phi_{\lambda,w}'(\langle
		z,w\rangle)w\otimes 
		wd\nu(w),
	\end{equation}
	for each $z\in\mathcal {D}$. Thus for each $y\in H'$,
	\begin{equation*}
		\langle y, dT_1(z)y\rangle=\int_{S^n\cap H'}\phi_{\lambda,w}'(\langle
		z,w\rangle)\langle w,y\rangle^2d\nu(w).
	\end{equation*}
	Since $\phi_{\lambda,w}'>0$ and $\nu$ is not concentrated on the
	subspace  $H'$, it follows that the matrix $dT_1(z)$
	is positive definite for each $z\in\mathcal {D}$. Hence, the mean
	value theorem shows that $T_1:\mathcal
	{D}\rightarrow H'$ is injective.
	\par
	In light of \eqref{equ-13}, let $t=\langle
	z,w\rangle$ for $z\in\mathcal {D}$ and $w\in
	\textrm{supp}\,\nu$. Integrating \eqref{equ-12} with respect
	to $\nu$ and by the fact that $\nu(S^n\cap H')=k+1$, we get
	\begin{align}\label{equ-17}
		&\exp\int_{S^n\cap H'}\Big(-\frac{\langle z,w\rangle^2}{2}+\big\langle 
		w,(\lambda-\sqrt{n+1})e_{n+1}\big\rangle\langle
		z,w\rangle\Big)d\nu(w)\notag\\
		&=\exp\int_{S^n\cap 
		H'}\Big(\log\frac{G_{\lambda,w}}{\sqrt{\pi}}-\phi_{\lambda,w}(\langle 
		z,w\rangle)^2
		+\log\phi_{\lambda,w}'(\langle
		z,w\rangle)\Big)d\nu(w)\notag\\
		&=\Big(\frac{1}{\sqrt{\pi}}\Big)^{k+1}A_\lambda\exp\int_{S^n\cap H'}
		\Big(-\phi_{\lambda,w}(\langle z,w\rangle)^2
		+\log\phi_{\lambda,w}'(\langle z,w\rangle)\Big)d\nu(w),
	\end{align}
where \begin{equation*}
	A_\lambda=\exp\int_{S^n\cap H'}\log 
	G_{\lambda,w}d\nu(w).
\end{equation*}
	\par
	Integrating \eqref{equ-17} on $\mathcal {D}$, from \eqref{equ-11},
	the Ball-Barthe inequality \eqref{e-5}, \eqref{equ-16}, Lemma 
	\ref{lem-1}, and
	making the change of variables $y=T_1z$, we have
	\begin{align*}
		&\int_{\mathcal {D}}\exp\bigg(-\int_{S^n\cap H'}\Big(\frac{\langle 
		z,w\rangle^2}{2}-\big\langle 
		w,(\lambda-\sqrt{n+1})e_{n+1}\big\rangle\langle
		z,w\rangle\Big)d\nu(w)\bigg)dz\\
		&=\Big(\frac{1}{\sqrt{\pi}}\Big)^{k+1}A_\lambda\int_{\mathcal {D}}
		\exp\Big(-\int_{S^n\cap H'}\phi_{\lambda,w}(\langle z,w\rangle)^2
		-\log\phi_{\lambda,w}'(\langle
		z,w\rangle)d\nu(w)\Big)dz\\
		&=\Big(\frac{1}{\sqrt{\pi}}\Big)^{k+1}A_\lambda\int_{\mathcal {D}}
		\exp\Big(-\int_{S^n\cap H'}\phi_{\lambda,w}(\langle z,w\rangle)^2
		d\nu(w)\Big)\exp\Big(\int_{S^n\cap H'}\log\phi_{\lambda,w}'(\langle 
		z,w\rangle)
		d\nu(w)\Big)dz\\
		&\leq\Big(\frac{1}{\sqrt{\pi}}\Big)^{k+1}A_\lambda\int_{\mathcal {D}}
		\exp\Big(-\int_{S^n\cap H'}\phi_{\lambda,w}(\langle z,w\rangle)^2
		d\nu(w)\Big)|dT_1(z)|dz\\
		&\leq\Big(\frac{1}{\sqrt{\pi}}\Big)^{k+1}A_\lambda\int_{\mathcal {D}}
		e^{-|T_1z|^2}|dT_1(z)|dz\\
		&\leq\Big(\frac{1}{\sqrt{\pi}}\Big)^{k+1}A_\lambda\int_{\mathbb{R}^{k+1}}
		e^{-|y|^2}dy\\
		&=A_\lambda.
	\end{align*}
	\par
	On the other hand, since $\nu$ is isotropic on $S^n\cap H'$, from
	\eqref{equ-10} and the definition of $\mathcal {D}$, we obtain, for $z=(x,r)\in\mathcal {D}$,
	\begin{align*}
		&\int_{\mathcal {D}}\exp\bigg(-\int_{S^n\cap H'}\Big(\frac{\langle 
			z,w\rangle^2}{2}-\big\langle 
		w,(\lambda-\sqrt{n+1})e_{n+1}\big\rangle\langle
		z,w\rangle\Big)d\nu(w)\bigg)dz\\
		&=\int_{\mathcal {D}}\exp\bigg(-\int_{S^n\cap H'}\Big(\frac{\langle 
			z,w\rangle^2}{2}-(\lambda-\sqrt{n+1})\langle z,\langle w,e_{n+1}\rangle w\rangle\Big)d\nu(w)\bigg)dz\\
		&=\int_{\mathcal {D}}\exp\bigg(-\int_{S^n\cap H'}\frac{\langle z,w\rangle^2}{2}d\nu(w)
		+(\lambda-\sqrt{n+1})\Big\langle z,
		\int_{S^n\cap H'}w\langle w,e_{n+1}\rangle d\nu(w)\Big\rangle\bigg)dz\\
		&=\int_{\mathcal {D}}\exp\Big(-\frac{|z|^2}{2}
		+(\lambda-\sqrt{n+1})\langle z,
		e_{n+1}\rangle\Big)dz\\
		&=\int_0^{\infty}\int_{\frac{r}{\sqrt{n}}\mathrm{relint}(C^{\circ}\cap H)}
		e^{-\frac{r^2}{2}+(\lambda-\sqrt{n+1})r}e^{-\frac{x^2}{2}}dxdr\\
		&=(2\pi)^{\frac{k}{2}}\int_0^{\infty}e^{-\frac{r^2}{2}+(\lambda-\sqrt{n+1})r}
		\gamma_k\Big(\frac{r}{\sqrt{n}}(C^{\circ}\cap H)\Big)dr.
	\end{align*}
	Therefore, we  prove that
	\begin{equation}\label{e-17}
		\int_0^{\infty}e^{-\frac{r^2}{2}+(\lambda-\sqrt{n+1})r}
		\gamma_k\Big(\frac{r}{\sqrt{n}}(C^{\circ}\cap 
		H)\Big)dr\leq(2\pi)^{-\frac{k}{2}}A_\lambda.
	\end{equation}
which, by \eqref{e-20} and Lemma  \ref{l-1}, implies that 
	\begin{equation}\label{e-14}
	W(\mathrm{P}_HC)
	\geq\sqrt{\frac{k}{n}}W(\triangle_k).
\end{equation}	
	\par
	Next, we will show that the equality of the above inequality holds
	if and only if $\mathrm{P}_HC=\sqrt{\frac{k}{n}}\triangle_k$. 
	Actually, 
	suppose that there is
	equality in \eqref{e-14}. Since $\nu$ is isotropic on $S^n\cap H'$, and 
	hence  
	is not concentrated
	on a proper subspace of $H'$, there exists a basis
	$$\{w_1,\ldots, w_{k+1}: w_{i}\in \textrm{supp}\,\nu, \ i=1,\ldots,k+1\}$$
	of $H'$. We will show that $\textrm{supp}\,\nu=\{w_1,\ldots, w_{k+1}\}.$  Assume that $\{w_0,
	w_2,\ldots, w_{k+1}\}$ is another basis of
	$H'$, where $w_0\in\textrm{supp}\,\nu$
	and $w_0=a_1w_1+\cdots+a_{k+1}w_{k+1}$ such
	that at least one coefficient, say $a_1$, is not zero. Then the
	equality conditions of the Ball-Barthe inequality \eqref{e-5} imply that
	\begin{equation*}
		\phi_{\lambda,w_1}'(\langle 
		z,w_1\rangle)\cdots\phi_{\lambda,w_{k+1}}'(\langle
		z,w_{k+1}\rangle)=\phi_{\lambda,w_0}'(\langle 
		z,w_0\rangle)\phi_{\lambda,w_2}'(\langle 
		z,w_2\rangle)\cdots\phi_{\lambda,w_{k+1}}'(\langle
		z,w_{k+1}\rangle),
	\end{equation*}
	for all $z\in\mathcal {D}$. Since $\phi_{\lambda,w}'>0$, we have
	\begin{equation*}
		\phi_{\lambda,w_1}'(\langle z,w_1\rangle)=\phi_{\lambda,w_0}'(\langle
		z,w_0\rangle),
	\end{equation*}
	for all $z\in\mathcal {D}$. Obviously, by  \eqref{equ-12}, the function 
	$\phi_{\lambda,w}$ is of class $C^2$. Thus, differentiating both sides 
	with respect	to $z$ shows that 
	\begin{equation*}
		\phi_{\lambda,w_1}''(\langle 
		z,w_1\rangle)w_1=\phi_{\lambda,w_0}''(\langle
		z,w_0\rangle)w_0,
	\end{equation*}
	for all $z\in\mathcal {D}$, which implies that
	$w_0=w_1$ since $\phi''(\langle
	z,w_1\rangle)\neq0$ for some $z\in\mathcal {D}$. By \eqref{e-4}, we see that $\nu$ is supported inside an open hemisphere on $S^n\cap H'$, which 
	gives $w_0=w_1$. Therefore, the assumption of
	equality implies that
	$$\textrm{supp}\,\nu=\{w_1,\ldots, w_{k+1}\}.$$
	Since $\nu$ is isotropic, it follows that
	$\{w_1,\ldots, w_{k+1}\}$ is an orthogonal basis of
	$H'$ (see, e.g., \cite{Lutwak-1,Lutwak-2}). Thus, 
		\begin{equation*}
		\langle w_i,
		w_j\rangle=0,	\ \	\ 1\leq i\neq j\leq k+1.
	\end{equation*}
	That is, by \eqref{e-4},  $$\langle\mathrm{P}_Hu_i, 
	\mathrm{P}_Hu_j\rangle=-\frac1n, $$ whenever 
		$\mathrm{P}_Hu_i\neq \mathrm{P}_Hu_j$, where $u_i,u_j\in 
	\mathrm{supp}\, \mu\setminus H^{\perp}$,
which is equivalent to   
	\begin{equation}\label{e-15}
	\Big\langle\mathrm{P}_H\sqrt{\frac nk}u_i, 
	\mathrm{P}_H\sqrt{\frac nk}u_j\Big\rangle=-\frac1k,
	\end{equation}
whenever 
$\mathrm{P}_Hu_i\neq \mathrm{P}_Hu_j$, where $u_i,u_j\in 
\mathrm{supp}\, \mu\setminus H^{\perp}$. Meanwhile, 
the Cauchy-Schwartz inequality and \eqref{e-6} yield that the equality of  inequality 
\eqref{e-7} holds if and only if  
$\|\mathrm{P}_{H'}v(u)\|=\sqrt{\frac{k+1}{n+1}}$. Thus, by \eqref{e-4}, we have
$\|\mathrm{P}_{H}u\|=\sqrt{\frac{k}{n}}$ for each $u\in 
\mathrm{supp}\, \mu\setminus H^{\perp}$, which means that 
$\mathrm{P}_H\sqrt{\frac nk}u$ is an unit vector in $H$ for each $u\in 
\mathrm{supp}\, \mu\setminus H^{\perp}$.
	Recall that 
	\begin{align*}
		\sqrt{\frac nk}\mathrm{P}_HC
		&=\sqrt{\frac nk}\mathrm{conv}\{\mathrm{P}_Hu: u\in 
		\mathrm{supp}\,\mu\setminus  H^{\perp}\}\\
		&=\mathrm{conv}\Big\{\mathrm{P}_H\sqrt{\frac nk}u: u\in 
		\mathrm{supp}\,\mu\setminus 
		H^{\perp}\Big\}.
	\end{align*}
	 Together with \eqref{e-15}, we have $\sqrt{\frac nk}\mathrm{P}_HC$
	is a regular simplex $\triangle_k$ inscribed in
	$S^{k-1}\subset H$, which gives
	that $\mathrm{P}_HC=\sqrt{\frac{k}{n}}\triangle_k$.
	\par
	On the other hand, if  $\mathrm{P}_HC=\sqrt{\frac{k}{n}}\triangle_k$, then 
	by the definition of $\mathrm{P}_HC$,  
	\begin{equation}\label{e-18}
		\Big(\mathrm{P}_H\sqrt{\frac 
	nk}u\Big)_{u\in 
		\mathrm{supp}\, \mu\setminus H^{\perp}} \ \ \mathrm{ are \ all 
	\ unit \ vectors \ in } \ H
	\end{equation}  with identity \eqref{e-15}. By \eqref{e-4}, we must 
	have	
	\begin{equation*}
		\langle w_i,
		w_j\rangle=0,	\ \	\ 1\leq i\neq j\leq k+1.
	\end{equation*}
Thus, $\{w_1,\ldots, w_{k+1}\}=\textrm{supp}\,\nu$ is an orthogonal 
basis of	$H'$.  From the fact that
$\sum_{i=1}^{k+1}w_i\otimes w_i=I_{k+1}$ and
$\langle w_i, w_j\rangle=\delta_{ij}$, we see that equalities in
the Ball-Barthe inequality and Lemma \ref{lem-1}
hold. Since $T_1: \mathcal {D}\rightarrow\mathbb{R}^{k+1}$ is 
globally
1-1, it follows that  equality in \eqref{e-17} holds. Meanwhile,  \eqref{e-18}  
implies that 
	$\|\mathrm{P}_{H}u\|=\sqrt{\frac{k}{n}}$, and thus 
	$\|\mathrm{P}_{H'}v(u)\|=\sqrt{\frac{k+1}{n+1}}$, a constant for each 
	$u\in 	\mathrm{supp}\, \mu\setminus H^{\perp}$. Therefore, by the equality 
	conditions of  Jensen's 
	inequality, equality in  \eqref{e-16}, and hence in \eqref{e-10}, holds.  
 The desired result immediately follows.
\end{proof}

Theorem \ref{thm-4-1} can be stated in terms of the $\ell$-norm of  convex bodies by taking account of \eqref{e-12} and \eqref{e-28}.
\begin{thm}\label{t-6}
	Let $H\in G_{n,k}$ and let
	$\triangle_k$ be a regular $k$-simplex inscribed in $B_2^k$. If $\mu$ is an 
	isotropic measure on $S^{n-1}$ whose centroid is at
	the origin, then
	\begin{equation*}
		\ell(C^\circ\cap H)
		\geq\sqrt{\frac{n}{k}}\ell(\triangle_k^\circ),
	\end{equation*}	
	where $C=\mathrm{conv}\{\mathrm{supp}\,\mu\}.$ There is equality if and
	only if $C^\circ\cap H=\sqrt{\frac{n}{k}}\triangle_k^\circ$.
\end{thm}
The case $k=n$ of Theorem \ref{t-6} was proved in \cite{LL}.

Let $K$ be a convex body in $\mathbb{R}^n$ in the L\"{o}wner position. Then by the John theorem,
 there exist $m\geq n$ contact points of $K$
and $S^{n-1}$ forming an isotropic measure $\mu$; i.e.,
$\textrm{supp}\,\mu=\{u_1,\ldots,u_m\}$. So  we have $$K\supseteq C=\textrm{conv}\{u_1,\ldots,u_m\},$$ which
implies $\mathrm{P}_HK\supseteq \mathrm{P}_HC$. Therefore, by Theorem
\ref{thm-4-1}, we obtain that
$$W(\mathrm{P}_HK)\geq W(\mathrm{P}_HC)\geq \sqrt{\frac{k}{n}}W(\triangle_k),$$ with
equality if and only if $\mathrm{P}_HK=\sqrt{\frac{k}{n}}\triangle_k$. This inequality was proved
by Alonso-Guti\'{e}rrez and Brazitikos \cite{Alonso-S}, and  by Schmuckenschl\"{a}ger \cite{Schmuckenschlager} for $k=n$.

\section{The symmetric case}

For $H\in G_{n,k}$, if $\mu$ is an isotropic measure  on $S^{n-1}$, then we define  the Borel measure $\bar{\mu}$  on $S^{n-1}\cap H$  by
\begin{equation}\label{e-24}
	\bar{\mu}(A)=\int_{S^{n-1}\setminus H^{\perp}}\textbf{1}_A\Big(\frac{\textrm{P}_Hu}
	{\|\textrm{P}_Hu\|}\Big)\|\textrm{P}_Hu\|^2d\mu(u)
\end{equation}
for Borel sets $A\subset S^{n-1}\cap H$. Thus, for an arbitrary
$y\in H$, we have
\begin{align*}
	\int_{S^{n-1}\cap
		H}(y\cdot w)^2d\bar{\mu}(w)&=\int_{S^{n-1}\setminus H^{\perp}}\Big( y\cdot\frac{\textrm{P}_Hu}{\|\textrm{P}_Hu\|}\Big)^2\|\textrm{P}_Hu\|^2d\mu(u)
	\notag\\
	&=\int_{S^{n-1}}(y\cdot u)^2d\mu(u)=\|y\|^2.
\end{align*}
Hence, $\bar{\mu}$ is isotropic on $S^{n-1}\cap H$, and
\begin{equation}\label{e-25}
	\bar{\mu}(S^{n-1}\cap H)=\int_{S^{n-1}\setminus H^{\perp}}\|\textrm{P}_Hu\|^2d\mu(u)=k.
\end{equation}

\begin{thm}\label{t-3}
		If $H\in
		G_{n,k}$ and $\mu$ is an even isotropic measure on $S^{n-1}$, then
	\begin{equation*}
		W(\mathrm{P}_HC)\geq \sqrt{\frac kn}W(B_{1}^k).
	\end{equation*}
	with equality if and only if $\mathrm{P}_HC=\sqrt{\frac kn}B_{1}^k$.
	\end{thm}
\begin{proof}
Let $\tau>0$ be a real number, and	for $u\in \mathrm{supp}\,\mu$, let
	\begin{equation}\label{equ-27}
		\Gamma_{\tau,u}=\frac{1}{\sqrt{2\pi}}\int_{\mathbb{R}}\textbf{1}_{\big[-\frac{\tau}{\|\mathrm{P}_Hu\|},\frac{\tau}{\|\mathrm{P}_Hu\|}\big]}(s)e^{-\frac{s^2}{2}}ds.
	\end{equation}
	Define the smooth strictly increasing function $\phi_{\tau,u}:
	\big(-\frac{\tau}{\|\mathrm{P}_Hu\|},\frac{\tau}{\|\mathrm{P}_Hu\|}
	\big)\rightarrow\mathbb{R}$ by
	\begin{equation*}
		\frac{1}{\Gamma_{\tau,u}\sqrt{2\pi}}
		\int_{-\infty}^t\textbf{1}_{\big[-\frac{\tau}{\|\mathrm{P}_Hu\|},
			\frac{\tau}{\|\mathrm{P}_Hu\|}\big]}(s)e^{-\frac{s^2}{2}}ds
		=\frac{1}{\sqrt{\pi}}\int_{-\infty}^{\phi_{\tau,u}(t)}e^{-s^2}ds.
	\end{equation*}
	Then $\phi_{\tau,u}'>0$, and  differentiating both sides gives
	\begin{equation}\label{equ-28}
		\log \textbf{1}_{\big[-\frac{\tau}{\|\mathrm{P}_Hu\|},
			\frac{\tau}{\|\mathrm{P}_Hu\|}\big]}(t)-\frac{t^2}{2}
		=\log(\sqrt{2}\Gamma_{\tau,u})-\phi_{\tau,u}(t)^2+
		\log\phi_{\tau,u}'(t),
	\end{equation}
	for all $t\in\big[-\frac{\tau}{\|\mathrm{P}_Hu\|},
	\frac{\tau}{\|\mathrm{P}_Hu\|}\big]$. 
	\par
	Since $\mu$ is an even isotropic measure on $S^{n-1}$,  by \eqref{equ-38}, for $\tau>0$, we have
	\begin{align*}
		\textrm{relint}(\tau(C^{\circ}\cap H))&=\{x\in H: |\langle x,\mathrm{P}_Hu\rangle|<\tau \ \textrm{for all} \ u\in \textrm{supp}\,\mu\setminus H^\perp\}\\
		&=\bigg\{x\in H: \Big|\Big\langle x,\frac{\mathrm{P}_Hu}{\|\mathrm{P}_Hu\|}\Big\rangle\Big|<
		\frac{\tau}{\|\mathrm{P}_Hu\|} \ \textrm{for all} \ u\in \textrm{supp}\,\mu\setminus H^\perp\bigg\}.
	\end{align*}
	Then for each $x\in\textrm{relint}(\tau C^{\circ})$
	\begin{equation}\label{equ-29}
		\exp\bigg(\int_{S^{n-1}\setminus H^\perp}\log \textbf{1}_{\big[-\frac{\tau}{\|\mathrm{P}_Hu\|},
			\frac{\tau}{\|\mathrm{P}_Hu\|}\big]}
		\Big(\Big\langle
		x,\frac{\mathrm{P}_Hu}{\|\mathrm{P}_Hu\|}\Big\rangle\Big)
		\|\mathrm{P}_Hu\|^2d\mu(u)\bigg)=1.
	\end{equation}
	\par
	Define $T: \textrm{relint}(\tau(C^{\circ}\cap H))\rightarrow H$ by
	\begin{equation}\label{equ-30}
		Tx=\int_{S^{n-1}\setminus H^\perp}\mathrm{P}_Hu\phi_{\tau,u}\Big(\Big\langle x,\frac{\mathrm{P}_Hu}{\|\mathrm{P}_Hu\|}\Big\rangle\Big)
		\|\mathrm{P}_Hu\|d\mu(u).
	\end{equation}
	Note that for all $x\in\textrm{relint}(\tau(C^{\circ}\cap H))$ and all $u\in
	\textrm{supp}\,\mu\setminus H^\perp$, $\big\langle x,\frac{\mathrm{P}_Hu}{\|\mathrm{P}_Hu\|}\big\rangle$ is in the domain of $\phi_{\tau,u}$.
	It follows that
	\begin{equation}\label{equ-31}
		dT(x)=\int_{S^{n-1}\setminus H^\perp}\mathrm{P}_Hu\otimes\mathrm{P}_Hu\phi_{\tau,u}'\Big(\Big\langle x,\frac{\mathrm{P}_Hu}{\|\mathrm{P}_Hu\|}\Big\rangle\Big)
		d\mu(u).
	\end{equation}
	Since $\phi_{\tau,u}'>0$, the matrix $dT (x)$ is positive definite for each
	$x\in\textrm{relint}(\tau(C^{\circ}\cap H))$. Hence, the transformation $T$:
	relint$(\tau(C^{\circ}\cap H))\rightarrow H$ is  injective.
	\par
  The following two inequalities are the direct consequences of Lemmas  \ref{l-6} and \ref{lem-1}:
 	\begin{equation}\label{e-26}
	\exp\bigg(\int_{S^{n-1}\setminus H^\perp}\log\phi_{\tau,u}'\Big(\Big\langle x,\frac{\mathrm{P}_Hu}
	{\|\mathrm{P}_Hu\|}\Big\rangle\Big)\|\mathrm{P}_Hu\|^2d\mu(u)\bigg)\leq |dT(x)|,	
	\end{equation}
	and
	\begin{equation}\label{e-27}
	\int_{S^{n-1}\setminus H^\perp}\phi_{\tau,u}\Big(\Big\langle x,\frac{\mathrm{P}_Hu}{\|\mathrm{P}_Hu\|}\Big\rangle\Big)^2
	\|\mathrm{P}_Hu\|^2d\mu(u)\geq \|Tx\|^2.	
	\end{equation}

	From \eqref{e-2} we get, for $x\in H$,
\begin{equation}\label{equ-32}
	e^{-\frac{|x|^2}{2}}=\exp\Big(-\frac12\int_{S^{n-1}\setminus H^\perp}\langle
	x,\mathrm{P}_Hu\rangle^2d\mu(u)\Big).
\end{equation}
	Therefore, by \eqref{equ-29}, \eqref{equ-28}, \eqref{e-22},
	\eqref{e-26},  \eqref{e-27}, and making the change of variable $y=Tx$,
 we have
	\begin{align*}
		&\gamma_k(\tau(C^{\circ}\cap H))\\&=\int_{\textrm{relint}(\tau(C^{\circ}\cap H))}d\gamma_k(x)\\&=(2\pi)^{-\frac{k}{2}}\int_{\textrm{relint}(\tau(C^{\circ}\cap H))}e^{-\frac{|x|^2}{2}}dx\\
		&=(2\pi)^{-\frac{k}{2}}\int_{\textrm{int}(\tau(C^{\circ}\cap H))}\exp\Bigg(\int_{S^{n-1}\setminus H^\perp}\bigg(\log\textbf{1}_{\big[-\frac{\tau}{\|\mathrm{P}_Hu\|},
			\frac{\tau}{\|\mathrm{P}_Hu\|}\big]}
		\Big(\Big\langle
		x,\frac{\mathrm{P}_Hu}{\|\mathrm{P}_Hu\|}\Big\rangle\Big)\\
		& \ \ \ \ \ \ \ \ \ \ \ \ \ \ \ \ \ \ \ \ \ -\frac{1}{2}\Big\langle x,\frac{\mathrm{P}_Hu}{\|\mathrm{P}_Hu\|}\Big\rangle^2\bigg)
		\|\mathrm{P}_Hu\|^2d\mu(u)\Bigg)dx\\	
		&=(2\pi)^{-\frac{k}{2}}
		\int_{\textrm{relint}(\tau(C^{\circ}\cap H))}
		\exp\Big(\int_{S^{n-1}\setminus H^\perp}\log(\sqrt{2}\Gamma_{\tau,u})\|\mathrm{P}_Hu\|^2d\mu(u)\Big)\\
		&\ \ \ \ \ \ \ \ \ \ \ \ \ \ \ \ \ \ \ \ \	\times\exp\bigg(\int_{S^{n-1}\setminus H^\perp}-\phi_{\tau,u}\Big(\Big\langle x,\frac{\mathrm{P}_Hu}{\|\mathrm{P}_Hu\|}\Big\rangle\Big)^2
		\|\mathrm{P}_Hu\|^2d\mu(u)\bigg)\\
		& \ \ \ \ \ \ \ \ \ \ \ \ \ \ \ \ \ \ \ \ \ \ \ \ \ \ \  \times\exp\bigg(\int_{S^{n-1}\setminus H^\perp}\log\phi_{\tau,u}'\Big(\Big\langle x,\frac{\mathrm{P}_Hu}{\|\mathrm{P}_Hu\|}\Big\rangle\Big)
		\|\mathrm{P}_Hu\|^2d\mu(u)\bigg)dx\\
		&\leq\pi^{-\frac{k}{2}}
		\int_{\textrm{relint}(\tau(C^{\circ}\cap H))}
		\exp\Big(\int_{S^{n-1}\setminus H^\perp}\log(\Gamma_{\tau,u})\|\mathrm{P}_Hu\|^2d\mu(u)\Big)\\
		&\ \ \ \ \ \ \ \ \ \ \ \ \ \ \ \ \ \ \ \ \	\times\exp\bigg(\int_{S^{n-1}\setminus H^\perp}-\phi_{\tau,u}\Big(\Big\langle x,\frac{\mathrm{P}_Hu}{\|\mathrm{P}_Hu\|}\Big\rangle\Big)^2
		\|\mathrm{P}_Hu\|^2d\mu(u)\bigg)|dT(x)|dx\\
		&\leq\pi^{-\frac{k}{2}}\exp\Big(\int_{S^{n-1}\setminus H^\perp}\log(\Gamma_{\tau,u})\|\mathrm{P}_Hu\|^2d\mu(u)\Big)
		\int_{\textrm{relint}(\tau(C^{\circ}\cap H))}e^{-|Tx|^2}|dT(x)|dx\\
		&\leq\pi^{-\frac{k}{2}}\exp\Big(\int_{S^{n-1}\setminus H^\perp}\log(\Gamma_{\tau,u})\|\mathrm{P}_Hu\|^2d\mu(u)\Big)
		\int_{H}e^{-|y|^2}dy\\
		&=\exp\Big(\int_{S^{n-1}\setminus H^\perp}\log(\Gamma_{\tau,u})\|\mathrm{P}_Hu\|^2d\mu(u)\Big)\\
		&=\exp\bigg(\int_{S^{n-1}\setminus H^\perp}\log\Big(\frac{1}{\sqrt{2\pi}}\int_{\mathbb{R}}1_{\big[-\frac{\tau}
			{\|\mathrm{P}_Hu\|},	\frac{\tau}{\|\mathrm{P}_Hu\|}\big]}(s)e^{-\frac{s^2}{2}}ds\Big)
			\|\mathrm{P}_Hu\|^2d\mu(u)\bigg)\\
		&=\exp\left[\frac 1k\int_{S^{n-1}\setminus H^\perp}\log\bigg(\gamma_1\Big(\frac{\tau}{\|\mathrm{P}_Hu\|}[-e_1,
			e_1]\Big)\bigg)
		\|\mathrm{P}_Hu\|^2d\mu(u)\right]^k	.
	\end{align*}
Since $\gamma_1$ is log-concave and  $\frac 1k\|\mathrm{P}_Hu\|^2d\mu(u)$ is a probability measure on $S^{n-1}\setminus H^\perp$ (by \eqref{e-22}),  by H\"{o}lder's inequality, we have 
\begin{align*}
&\exp\bigg[\frac 1k\int_{S^{n-1}\setminus H^\perp}\log\bigg(\gamma_1\Big(\frac{\tau}{\|\mathrm{P}_Hu\|}[-e_1,
e_1]\Big)\bigg)
\|\mathrm{P}_Hu\|^2d\mu(u)\bigg]^k\\
&\leq\gamma_1\bigg(\frac 1k\int_{S^{n-1}\setminus H^\perp}\frac{\tau}{\|\mathrm{P}_Hu\|}[-e_1,
e_1]\|\mathrm{P}_Hu\|^2d\mu(u)\bigg)^k\\
&=\gamma_1\bigg(\frac \tau k\int_{S^{n-1}\setminus H^\perp}\|\mathrm{P}_Hu\|d\mu(u)[-e_1,
e_1]\bigg)^k\\
&=\gamma_k\bigg(\frac \tau k\int_{S^{n-1}\setminus H^\perp}\|\mathrm{P}_Hu\|d\mu(u)B_{\infty}^k\bigg),		
\end{align*}
	where $B_{\infty}^k$ is a unit cube in $\mathbb{R}^k$.
Meanwhile, by H\"{o}lder's inequality, \eqref{equ-5}, and \eqref{e-22}, we also 
have
\begin{align}\label{e-23}
	\frac{1}{k}\int_{S^{n-1}\setminus 
		H^{\perp}}\|\mathrm{P}_Hu\|d\mu(u)&\leq
	\frac{1}{k}\Big(\int_{S^{n-1}\setminus 
		H^{\perp}}d\mu(u)\Big)^{\frac12}\Big(\int_{S^{n-1}\setminus 
		H^{\perp}}\|\mathrm{P}_{H}u\|^2d\mu(u)\Big)^{\frac12}
	\notag\\
	&\leq
	\frac{1}{k}\Big(\int_{S^{n-1}}d\mu(u)\Big)^{\frac12}
	\Big(\int_{S^{n-1}\setminus 
		H^{\perp}}\|\mathrm{P}_{H}u\|^2d\mu(u)\Big)^{\frac12}
	\notag\\
	&=\frac{\sqrt{nk}}{k}=\sqrt{\frac nk}.	
\end{align}
	Hence,  we obtain
	\begin{equation*}
		\gamma_k(\tau(C^{\circ}\cap H))\leq\gamma_k\Big(\tau\sqrt{\frac nk}B_{\infty}^k\Big),
	\end{equation*}
which implies that
\begin{equation*}
	\int_0^\infty(1-\gamma_k(\tau(C^{\circ}\cap H)))d\tau\geq\int_0^\infty\bigg(1-\gamma_k\Big(\tau\sqrt{\frac nk}B_{\infty}^k\Big)\bigg)d\tau.
\end{equation*}
	for all $\tau>0$. Thus, by \eqref{e-12} and \eqref{equ-7}, we finally get
	\begin{equation*}
		W(\mathrm{P}_HC)\geq W\Big(\sqrt{\frac kn}B_{1}^k\Big)=\sqrt{\frac kn}W(B_{1}^k).
	\end{equation*}
	
	Assume that equality holds in the above inequality. Since $\mu$ is even isotropic, it follows that $\bar{\mu}$ defined in \eqref{e-24} is also even isotropic on $S^{n-1}\cap H$, and thus $\bar{\mu}$ is not concentrated on any subsphere of
	$S^{n-1}\cap H$. Then there exist linearly independent $w_1,\ldots,
	w_{k}\in \textrm{supp}\,\bar{\mu}$ such that $\{\pm w_1,\ldots,\pm w_{k}\}\subseteq
	\textrm{supp}\,\bar{\mu},$  where $w_i=\frac{\mathrm{P}_Hu_i}{\|\mathrm{P}_Hu_i\|}$, $u_i\in \mathrm{supp}\, \mu\setminus H^\perp$, $i=0,1,\dots,k$. The argument, as the same as the proof in Theorem  \ref{thm-4-1}, yields that 	$\{w_1,\ldots,w_k\}=\Big\{\frac{\mathrm{P}_Hu_1}{\|\mathrm{P}_Hu_1\|},\dots,\frac{\mathrm{P}_Hu_k}{\|\mathrm{P}_Hu_k\|}\Big\}$ is in fact an orthogonal basis in $S^{n-1}\cap H$. 	
		Moreover,  the equality conditions of H\"{o}lder's inequality imply that equality in \eqref{e-23} holds if and only if $\|\mathrm{P}_Hu\|=\sqrt{\frac kn}$ for all $u\in \mathrm{supp}\, \mu\setminus H^\perp$ and
		 $\mu$ must concentrate on $S^{n-1}\setminus 
		 H^{\perp}$.  It means that 
		 $\mathrm{P}_H\sqrt{\frac nk}u$ is an unit vector in $H$ for each $u\in 
		 \mathrm{supp}\, \mu\setminus H^{\perp}$.
		 Thus, 
		 \begin{align*}
		 	\sqrt{\frac nk}\mathrm{P}_HC
		 	&=\sqrt{\frac nk}\mathrm{conv}\{\mathrm{P}_Hu: u\in 
		 	\mathrm{supp}\,\mu\setminus  H^{\perp}\}\\
		 	&=\mathrm{conv}\Big\{\mathrm{P}_H\sqrt{\frac nk}u: u\in 
		 	\mathrm{supp}\,\mu\setminus 
		 	H^{\perp}\Big\}\\
		 	&=\mathrm{conv}\Big\{\pm\mathrm{P}_H\sqrt{\frac nk}u_1,\ldots,\pm\mathrm{P}_H\sqrt{\frac nk}u_k\Big\},
		 \end{align*}
	which is $B_{1}^k$	in $H$. 
		
The if part is easy to check, and we omit the details.		
\end{proof}
\par
When $\mu$ is a discrete even isotropic measure, the case $k=n$ of Theorem \ref{t-3} was proved by Schechtman and Schmuckenschl\"{a}ger \cite{Schechtman}.

\begin{thm}\label{t-4}
	If $H\in
	G_{n,k}$ and $\mu$ is an even isotropic measure on $S^{n-1}$, then
	\begin{equation*}
		W(C^{\circ}\cap H)\leq\sqrt{\frac nk}W(B_{\infty}^k).
	\end{equation*}
 with equality if and only if $C^\circ\cap H=\sqrt{\frac nk}B_\infty^k$.
\end{thm}
\begin{proof}
	From \eqref{e-2} we have
	\begin{equation*}
		x=\int_{S^{n-1}\setminus H^{\perp}}\langle x,\mathrm{P}_Hu\rangle 
		\mathrm{P}_Hud\mu(u)
	\end{equation*}
	for all $x\in H$. Define a zonoid $Z$ in $H$ by
	\begin{equation*}
		h_Z(x)=\int_{S^{n-1}\setminus H^{\perp}}|\langle x,\mathrm{P}_Hu\rangle|
		d\mu(u).
	\end{equation*}
	From the definition of the polar body, it follows that
	$$Z^{\circ}=\{x\in H: h_Z(x)\leq1\}.$$
	\par
	Define a closure as
	\begin{equation*}
		M=\textrm{cl}\Big\{\int_{S^{n-1}\setminus H^{\perp}}\langle 
		x,\mathrm{P}_Hu\rangle
		\mathrm{P}_Hud\mu(u)\in H: h_Z(x)\leq1\Big\}.
	\end{equation*}
	It is easily shown that
	$M$ is a convex body in $H$ and $M=Z^{\circ}$. (The polarity is taken in 
	$H$.) Since
	$\mu$ is an even isotropic measure on $S^{n-1}$, from the definition
	of support function and \eqref{e-21}, we have, for $y\in H$,
	\begin{align*}
		h_M(y)&=\sup_{h_Z(x)\leq1}\Big|\Big\langle y,\int_{S^{n-1}\setminus 
		H^{\perp}}\langle x,\mathrm{P}_Hu\rangle
		\mathrm{P}_Hud\mu(u)\Big\rangle\Big|\\
		&=\sup_{h_Z(x)\leq1}\Big|\int_{S^{n-1}\setminus H^{\perp}}\langle 
		x,\mathrm{P}_Hu\rangle\langle y,\mathrm{P}_Hu\rangle
		d\mu(u)\Big|\\
		&\leq\sup_{h_Z(x)\leq1}\int_{S^{n-1}\setminus H^{\perp}}|\langle 
		x,\mathrm{P}_Hu\rangle||\langle
		y,\mathrm{P}_Hu\rangle|
		d\mu(u)\\
		&\leq\sup_{u\in \textrm{supp}\,\mu\setminus H^{\perp}}|\langle 
		y,\mathrm{P}_Hu\rangle|\\
		&=h_{\mathrm{P}_HC}(y),
	\end{align*}
	where $\mathrm{P}_HC=\mathrm{conv}\{\mathrm{P}_Hu: u\in 
	\mathrm{supp}\,\mu\setminus 
	H^{\perp}\}$. Hence $M\subseteq \mathrm{P}_HC$.
	Therefore, for $x\in H$,
	\begin{equation*}
		h_{C^{\circ}\cap H}(x)=h_{(\mathrm{P}_HC)^{\circ}}(x)\leq 
		h_{M^{\circ}}(x)=\int_{S^{n-1}\setminus H^{\perp}}|\langle 
		x,\mathrm{P}_Hu\rangle|
		d\mu(u).
	\end{equation*}
Recall that if $K$ is a convex body in $H$, then by  \eqref{e-11}, $$W(K)=\frac {1}{c_k}\int_{H}h_K(x)d\gamma_k(x).$$
	Integrating with respect to the Gaussian measure, using the Fubini
	theorem, the rotational invariance of Gaussian measure, and \eqref{e-23}, we get, 
	\begin{align*}
		W(C^{\circ}\cap H)&=\frac {1}{c_k}\int_{H}h_{C^{\circ}\cap H}(x)d\gamma_k(x)\\		
		&\leq\frac {1}{c_k}\int_{H}\int_{S^{n-1}\setminus 
		H^{\perp}}|\langle 
		x,\mathrm{P}_Hu\rangle|
		d\mu(u)d\gamma_k(x)\\
		&=\frac {1}{c_k}\int_{S^{n-1}\setminus 
			H^{\perp}}\int_{H}|\langle x,\mathrm{P}_Hu\rangle|
		d\gamma_k(x)d\mu(u)\\
		&=\frac {1}{c_k}\int_{S^{n-1}\setminus 
			H^{\perp}}\int_{H}\Big|\Big\langle 
			x,\frac{\mathrm{P}_Hu}{\|\mathrm{P}_Hu\|}\Big\rangle\Big|
		d\gamma_k(x)\|\mathrm{P}_Hu\|d\mu(u)\\
		&=\frac {1}{c_k}\int_{S^{n-1}\setminus 
			H^{\perp}}\int_{H}|\langle x,e_1\rangle|
		d\gamma_k(x)\|\mathrm{P}_Hu\|d\mu(u)\\
		&=\frac {1}{c_k}\int_{H}|\langle x,e_1\rangle|
		d\gamma_k(x)\int_{S^{n-1}\setminus 
			H^{\perp}}\|\mathrm{P}_Hu\|d\mu(u)\\
		&=\frac{1}{kc_k}\int_{H}\sum_{i=1}^k|\langle x,e_i\rangle|
		d\gamma_k(x)\int_{S^{n-1}\setminus 
			H^{\perp}}\|\mathrm{P}_Hu\|d\mu(u)\\
		&=\frac{1}{k}W(B_{\infty}^k)\int_{S^{n-1}\setminus 
			H^{\perp}}\|\mathrm{P}_Hu\|d\mu(u)\\
		&\leq\sqrt{\frac nk}W(B_{\infty}^k).
	\end{align*}

	Now we study the equality case. Assume that $W(C^{\circ}\cap H)\leq\sqrt{\frac nk}W(B_{\infty}^k)$.
	Then we must have
	\begin{equation*}
		h_{C^{\circ}\cap H}(x)=\int_{S^{n-1}\setminus H^{\perp}}|\langle x,\mathrm{P}_Hu\rangle|
		d\mu(u),
	\end{equation*}
	for all $x\in\mathbb{R}^n$. Thus, for $v\in \textrm{supp}\,\mu\cap H$, we
	get
	\begin{equation*}
		h_{C^{\circ}\cap H}(v)=\int_{S^{n-1}\setminus H^{\perp}}|\langle v,\mathrm{P}_Hu\rangle|
		d\mu(u).
	\end{equation*}
	Notice that for $v=\mathrm{P}_Hv\in \partial (\mathrm{P}_HC)$ (the boundary of $\mathrm{P}_HC$) and
	$$h_{C^{\circ}\cap H}(v)=\|v\|_{\mathrm{P}_HC}=\min\{t>0: v\in t(\mathrm{P}_HC)\}=1.$$
	Since $|\langle v,\mathrm{P}_Hu\rangle|\leq1$, it follows from \eqref{e-2} that
	\begin{equation*}
		1=\int_{S^{n-1}}|\langle v,\mathrm{P}_Hu\rangle|
		d\mu(u)\geq\int_{S^{n-1}}|\langle v,\mathrm{P}_Hu\rangle|^2
		d\mu(u)=|v|^2=1.
	\end{equation*}
	Necessarily, $|\langle v,\mathrm{P}_Hu\rangle|=1$, or $\langle
	v,\mathrm{P}_Hu\rangle\in\{-1,1\}$ for arbitrary $v\in \textrm{supp}\,\mu\cap H$ and $u\in \textrm{supp}\,\mu\setminus H^{\perp}$.
	This means that $\|\mathrm{P}_Hu\|=1$ and $\mu$ concentrates on $H \cup H^\perp$. Since $\mu$ is even and not concentrated on a proper subspace of
	$\mathbb{R}^n$, there must exist orthogonal unit vectors
	$v_1,\ldots,v_k$ in $H$ such that
	$$\textrm{supp}\,\mu\cap H=\{\pm v_1,\ldots, \pm v_k\}.$$
	Hence, by \eqref{e-23},
	\begin{equation*}
		\mu(S^{n-1}\cap H)=\mu(S^{n-1}\setminus H^\perp)=\sqrt{nk}.
	\end{equation*}
	Since $\mu$ is isotropic, it follows that
	$\mu(\{\pm v_i\})=\sqrt{\frac nk}$, $i=1,\dots,k$, and therefore  $C^\circ\cap H=\sqrt{\frac nk}B_\infty^k$.
	
	The if part is obvious by the definition of mean width. 
\end{proof}

Theorems \ref{t-3} and \ref{t-4} can be stated in terms of the $\ell$-norm of  convex bodies by taking account of \eqref{e-12} and \eqref{e-28}.
\begin{thm}\label{t-5}
If  $H\in
G_{n,k}$ and  $\mu$ is an even isotropic measure on $S^{n-1}$, then
$
\ell(C^\circ\cap H)\geq \sqrt{\frac nk}\ell(B_{\infty}^k)
$
with equality if and only if $C^\circ\cap H=\sqrt{\frac nk}B_{\infty}^k$; and 
$
\ell(\mathrm{P}_HC)\leq\sqrt{\frac kn}\ell(B_{1}^k)
$
with equality if and only if $\mathrm{P}_HC=\sqrt{\frac kn}B_1^k$.
\end{thm}

The case $k=n$ of Theorem \ref{t-5} was proved in \cite{LL}, and when $\mu$ is a discrete even isotropic measure, this situation was
remarked by Schechtman and Schmuckenschl\"{a}ger \cite{Schechtman},
and proved by Barthe \cite{Barthe}.

\bibliographystyle{amsalpha}

\end{document}